\documentclass[10pt,leqno]{article}

\usepackage{enumerate,xspace}
\usepackage{amsmath,amssymb,amsthm,wasysym}
\usepackage[all]{xy}
\usepackage{mathtools}
\usepackage{dsfont}
\usepackage{array}
\usepackage{tikz}
\usepackage{tikz-qtree}

\usepackage{hyperref} %hyperlink package
\hypersetup{
    colorlinks,
    citecolor=red,
    filecolor=red,
    linkcolor=blue,
    urlcolor=red
}

\newtheorem{observation}{Remark}[section]
\newtheorem{lemm}[observation]{Lemma}  %%share counter with remark
\newtheorem{theo}[observation]{Theorem}
\newtheorem*{theo*}{Theorem}
\newtheorem*{prop*}{Proposition}
\newtheorem{nota}[observation]{Notation}
\newtheorem{prop}[observation]{Proposition} 
\newtheorem{coro}[observation]{Corollary}
\newtheorem*{nota*}{Notation}
\theoremstyle{definition}
\newtheorem{defi}[observation]{Definition}
\theoremstyle{remark}
\newtheorem{rema}[observation]{Remark}
\newtheorem{ex}[observation]{Example}

\def\P{\mathcal P}
\def\Q{\mathcal Q}
\def\Com{\operatorname{Com}}
\def\uCom{\operatorname{uCom}}
\def\SuCom{\mathfrak{uCom}}

\def\SMagCom{\mathfrak{MagCom}}

\def\D{\operatorname{D}}
\def\SD{\mathfrak{D}}
\def\Lev{\operatorname{Lev}}
\def\SLev{\mathfrak{Lev}}
\def\C{\mathcal C}
\def\mod{\mathsf{mod}}
\def\alg{\mathsf{alg}}
\def\K{\mathcal K}
\def\Sym{\mathsf{Sym}}
\def\Hom{\mathsf{Hom}}
\def\Set{\mathsf{Set}}
 
\def\0{\mathbf 0}
\def\1{\mathbf 1}
\def\N{\mathbb N}
\def\Z{\mathbb Z}
\def\F{\mathbb F}

\def\t{\mathfrak t}

\def\id{\operatorname{id}}
\def\∑{\mathfrak S}
\def\Vect{\mathsf{Vect}_{\F}}
\def\Forget{\operatorname{Forget}}

\def\diag{\xymatrix}
\def\Dp{\D^{\pm}}
\def\Ind{\operatorname{Ind}}
\def\Res{\operatorname{Res}}
\def\A{\mathcal A}
\def\U{\mathcal U}

\def\scr{\mathcal}

\title{Unstable algebras over an operad II}
\author{Sacha Ikonicoff}
\begin{document}
    \maketitle
    \begin{abstract}
        We introduce a notion of unstable algebra over an operad in general characteristic. We show that the unstable algebra freely generated by an unstable module is itself a free algebra under suitable conditions. We introduce a family of ‘$q$-level' operads which allows us to identify unstable modules studied by Brown--Gitler, Miller and Carlsson to free unstable $q$-level algebras.  
    \end{abstract}
    \tableofcontents
    \section{Introduction}\label{sec:intro}
    In this article, we define unstable algebras over an operad in positive characteristic, and we characterise free unstable algebras. The mod $p$ Steenrod algebra $\A$ was introduced by Steenrod and Epstein \cite{SE} to study the stable operations of the classical cohomology functors with coefficients in the finite field $\F_p$ of order~$p$. Unstable modules are a class of (graded) $\A$-modules which satisfy a property called instability. A detailed survey of unstable modules and their properties can be found in Schwartz's textbook \cite{LS}. The principal source of examples of unstable modules is precisely the mod $p$ cohomology of topological spaces. These unstable modules coming from topology are endowed with an additional internal multiplication - the \textit{cup-product} - which is associative and commutative. The category $\U$ of unstable modules is understood fairly well. It has injective cogenerators, the Brown--Gitler modules, which can be obtained as the `Spanier-Whitehead dual' to the cohomology of a spectrum - the Brown--Gitler spectrum \cite{BG,HM}. Other injective objects of interest include the Carlsson modules~\cite{GC,LZ}, which are obtained as a limit of Brown--Gitler modules. The Carlsson modules were introduced in \cite{GC} to prove the Segal conjecture for Burnside rings of elementary abelian groups, and were used by Miller \cite{HM} to prove the Sullivan conjecture on maps from classifying spaces. Both the Brown--Gitler and the Carlsson modules are endowed with a multiplication which naturally arises when studying their structure. It has been shown that certain of these modules, with their respective multiplication, are in fact free objects in a specified category of (non-associative) algebras \cite{DD, SI2}.

    In \cite{SI2}, we define a notion of unstable algebra over an operad in characteristic two. For an operad $\P$ and a commutative operation $\star\in\P(2)$, a $\star$-unstable $\P$-algebra is an unstable module $A$ over the mod 2 Steenrod algebra, endowed with an action of $\P$ satisfying the following identity, called instability:
\[\forall x\in \A^i\quad Sq^ix=x\star x.\]
    When $\star$ is the operadic generator of the operad $\Com$ of commutative, associative (unital) algebras, we recover the classical definition of an unstable algebra.
    
    The main result of \cite{SI2} shows the existence, under some hypotheses, of an isomorphism of graded $\P$-algebras between the free $\star$-unstable $\P$-algebra generated by an unstable module $M$ and the free $\P$-algebra generated by the underlying graded vector space of the unstable module $\Sigma\Omega M$, where $\Sigma$ is the suspension functor and $\Omega$ is left adjoint to~$\Sigma$. The only condition that the operad $\P$ must satisfy for this result to hold is that the operation $\star$ has to be $\P$-central, which means that it satisfies a certain interchange relation with respect to all other operations in $\P$.

    Our aim is to extend the main result of \cite{SI2} to the odd characteristic case. We study modules and algebras over the Steenrod algebra of reduced $q$-th powers \cite{K1}. This algebra is generated by elements $P^i$ under certain relations called the Adem relations. We consider operads $\P$ in vector spaces equipped with a $q$-ary operation $\star\in\P(q)^{\∑_q}$. We identify $\P$-algebras in the symmetric monoidal category of modules over a bialgebra and, in particular, $\P$-algebras in unstable modules. A $\P$-algebra $A$ in unstable modules is called $\star$-unstable if it also satisfies:
    \[
        \forall x\in \A^i\quad P^ix=\star(\underbrace{x,\dots,x}_q).
    \]
    In contrast with \cite{SI2}, here we highlight the algebraic structures involved, and in particular the key notion of $\P$-centrality for the operation $\star$. An important feature of $\P$-centrality is that, if $\star$ is $\P$-central, then for any $\P$-algebra $A$, the linear map $a\mapsto \star(a,\dots,a)$, is a morphism of $\P$-algebras (the precise statement is contained in Lemma \ref{lemm:psiext}). We then build a functor $K_\P^{\star}$, which sends an unstable module $M$ to the free $\star$-unstable $\P$-algebra generated by $M$. We obtain the following:
    \begin{theo*}[Theorem \ref{theo:freeunst}]
       Let $\star\in\P(q)^{\∑_q}$ be a $\P$-central operation. For all connected reduced unstable module $M$, the underlying $\P$-algebra of $K_{\P}^\star(M)$ is the free algebra over the underlying vector space of $\Sigma\Omega M$.
   \end{theo*}
    Here, the conditions of being reduced and connected for an unstable module are defined in \ref{def:unstredconnsus}. 
    
    Note that the proof for this result is somewhat more straightforward than the one of the analogous result in characteristic $2$ given in \cite{SI2}. It relies on the algebraic and categorical framework described in the first part of this article, and more precisely on the results obtained in Section \ref{sec:centop} on central operations. A previous version of the present article contained a proof of this result which followed the steps of \cite{SI2}, and containing the additional results necessary to adapt the proof to the odd characteristic case. Following the suggestions of an anonymous referee, we choose here to present a somewhat more elegant, shorter proof, that encompass all positive characteristic cases.

    We then apply this result to identify some algebra structures on analogues of the Brown--Gitler modules, which are injective cogenerators of the category of unstable modules, and on analogues of the Brown--Gitler algebra, the Carlsson modules and the Carlsson algebra. We consider a $q$-ary operation on the Brown--Gitler modules and on the Carlsson modules, which corresponds to the classical multiplication studied by Carlsson \cite{GC} in characteristic 2, and by Miller \cite{HM} in characteristic~$p>2$. This multiplication of arity $q$ is fixed by the action of $\∑_q$, and satisfies an additional relation of strong commutativity, but it is not associative. In order to study this operation, we introduce an operad $\Lev_q$ whose algebras are called $q$-level algebras, and which generalises the operad $\Lev$ of level algebras of Chataur and Livernet \cite{CL}.

    We first define the operad $\Lev_q$ as a suboperad of a certain operad of ordered partitions. We then obtain a presentation of this operad which yields the following characterisation:
    \begin{theo*}[Theorem \ref{theo:Lev}]
        The operad $\Lev_q$ is generated by one element $\star\in\Lev_q(q)$, fixed under the action of $\∑_q$, and satisfying the unique following relation:
    \[(\star\circ_2\star)\circ_1\star=(\star\circ_2\star)\circ_1\star \cdot \sigma,\]
    where $\sigma\in\∑_{3q-2}$ is the transposition $(q\  q+1)$.
    \end{theo*}
    
    Identifying the structure of free unstable $q$-level algebras yields, for instance, the following result, which is analogous to a result of Davis \cite{DD} in characteristic 2, and gives the explicit algebraic structure for the Carlsson module of weight 1:
\begin{prop*}[Proposition \ref{prop:K1}]
 The Carlsson module of weight one $K(1)$, with its $\Lev_q$ operation, is isomorphic to the free $\star$-unstable $\Lev_q$-algebra generated by~$F(1)$.
\end{prop*}
We also use the naturality of our constructions to compare free unstable algebras over operads related by operadic morphisms.

The notions introduced in this article give an opportunity to revisit classical results on unstable modules and algebras from the point of view of operads. For instance, we have obtained results concerning free algebras generated by unstable modules equipped with an unstable ‘twisted' action which generalise results of Campbell--Selick \cite{CS} and Duflot--Kuhn--Winstead \cite{DKW}. These results are available in former versions of the present article, and we hope to publish them in future works. Other possible applications of these result stay open, including the link between unstable algebras over an operad and the analytic functors viewpoint of Hennes--Lannes--Schwartz \cite{HLS}, or with the generic representations viewpoint of \cite{K1,K3}.

    \begin{nota*}
        \begin{itemize}
            \item Our base field will be denoted by $\F$. Unless clearly defined otherwise, it has characteristic $p$ and order $q=p^\alpha$.
            \item The category of $\F$-vector space will be denoted by $\Vect$.
            \item The set $\{1,\dots,n\}$ will be denoted by $[n]$.
            \item The symmetric group on $n$ letters (or permutation group of $[n]$) will be denoted by $\∑_n$.
        \end{itemize}
    \end{nota*}
    % \ack
    \part{Algebraic background}
    In this first part of the article, we review the algebraic structures necessary to study unstable algebras over an operad. We introduce the notion of $\P$-central operation for an operad $\P$, and we define a family of operads $\Lev_q$, for which we obtain a presentation in Theorem \ref{theo:Lev}.
    \section{Recollections about operads}\label{sec:recop}
    
    In this section we present the notions of symmetric sequences, operads, and their algebras in a general category $\C$. We assume that the reader has some basic familiarity with these notions. Our main references on the subject are the two textbook accounts \cite{LV,BF2}, as well as the article \cite{GK}. Later in this article, we will be interested in operads and algebras in the category of sets, in the category $\Vect$ of $\F$-vector spaces, in the category of left module over a (cocommutative) $\F$-bialgebra, and in the category $\U(q)$ of unstable modules.

    For this section, fix a symmetric monoidal category $(\C,\otimes,I)$ with all small colimits and all small limits. Assume that $\otimes$ preserve colimits. In particular, if $G$ is a group, we can define the notion of a $G$-object in $\C$ as objects of $\C$ with an action of $G$. For a $G$-object $X$, we can define the object of $G$-orbits $X_G$, and if $H<G$ is a subgroup, we have a restriction functor $\Res_H^G$ from $G$-objects to $H$-objects, and this functor has a left adjoint $\Ind_h^G$ called induction, which can be obtain as a Kan extension. For more details on this, we refer the reader to \cite[Chapter 2]{BF2}.
    \begin{defi}[Symmetric sequences {\cite[Section 2.1]{BF2}}]
    A \textbf{symmetric sequence} $\mathcal M$ is a sequence $(\mathcal M(n))_{n\in\N}$ of objects in $\C$ such that, for all $n\in\N$, the symmetric group $\∑_n$ acts on $\mathcal M(n)$ on the right. $\mathcal M(n)$ is said to be the object of \textbf{arity} $n$ operations. Symmetric sequences form a category~$\Sym$. This category is endowed with a tensor product $\otimes_{\Sym}$ such that, if $\mathcal M$ and $\mathcal N$ are two symmetric sequences, then
    \[\left(\mathcal M\otimes_{\Sym}\mathcal N\right)(n)=\coprod_{i+j=n}\Ind_{\∑_i\times\∑_j}^{\∑_n}\mathcal M(i)\otimes \mathcal N(j),\]
    where $\Ind_{\∑_i\times\∑_j}^{\∑_n}$ denotes the induced representation from the Young subgroup $\∑_i\times\∑_j$ of the group~$\∑_n$.
    
    The category of symmetric sequences is endowed with another monoidal product $\circ$, called the \textbf{composition} of symmetric sequences, given by:
    \[\left(\mathcal M\circ\mathcal N\right)(n)=\coprod_{k\ge0}\mathcal M(k)\otimes_{\∑_k} (\mathcal N^{\otimes_{\Sym} k}(n)),\]
    with unit $I$ concentrated in arity 1. This unit is denoted by $I$ as well. When it is suitable to talk about elements, if $\nu\in\mathcal M(k)$, and $\xi_1,\dots,\xi_k\in\mathcal N$, let $(\nu;\xi_1,\dots,\xi_k)$ denote the element 
    \[
      \left[\nu\otimes \xi_1\otimes\dots\otimes \xi_k\right]_{\∑_k}\in\mathcal M\circ\mathcal N.
    \]
\end{defi}
\begin{defi}[Operads {\cite[Section 3.1]{BF2}}]
     An \textbf{operad} is a monoid object in the monoidal category of symmetric sequences~$(\Sym,\circ,I)$. For an operad $\P$, $\gamma_{\P}\colon\P\circ\P\to\P$ denotes its composition morphism, and $\1_{\P}:I\to\P$ its unit.
     
     For $\nu\in\P(n)$, $\xi_1,\dots,\xi_n\in\P$, let $\nu(\xi_1,\dots,\xi_n)$ denote the element 
    \[
      \gamma_\P(\nu;\xi_1,\dots,\xi_n)\in \P.
    \]
    The \textbf{partial compositions} in an operad $\P$ are defined by:
    \[\nu\circ_i\xi=\nu(1_\P,\dots,1_\P,\underbrace{\xi}_{i\mbox{\scriptsize{-th input}}},1_\P,\dots,1_\P),\] 
    for all $i\in\{1,\dots,m\}$, where $m$ is the arity of~$\nu$.

    For any operad $\P$, let $S(\P,-)$ denote the associated functor $\C\to\C$, defined on objects by:
    \[S(\P,X)=\coprod_{n\ge0}\P(n)\otimes_{\∑_n}X^{\otimes n}.\] 
    It can also be defined as $S(\P,X)=\P\circ X$ where $X$ is considered as a symmetric sequence concentrated in arity~$0$. 

    The unit and compositions in $\P$ induce a monad structure on $S(\P,-)$.
\end{defi}
\begin{defi}[Algebras over an operad {\cite[Section 3.2]{BF2}}]\label{defalgop}
    For an operad $\P$, a \textbf{$\P$-algebra} is an algebra over the monad~$S(\P,-)$. In other words, a $\P$-algebra is a pair $(A,\theta)$ where $A$ is an object of $\C$ and $\theta\colon S(\P,A)\to A$ is compatible with the composition and unit of~$\P$. 
    
    For $(A,\theta)$ a $\P$-algebra, $\nu\in\P(n)$ and $a_1,\dots,a_n\in A$, let $\mu(a_1,\dots,a_n)$ denote the element $\theta(\mu;a_1,\dots,a_n)\in A$. In the case where $a=a_1=\dots=a_n$, we use the notation $a^{\mu n}$ for the elements $(\mu;a,\dots,a)\in S(\P,A)$ and $\theta(\mu;a,\dots,a)\in A$ depending on the context.

    Let $\P_{\alg}$ denote the category of $\P$-algebras.
\end{defi}  
    \begin{ex}\label{ex:operads} Here are basic example in the category $\Set$ of sets, with tensor product given by the cartesian product:
        \begin{itemize}
            \item There is an operad $\SuCom$ such that, for all $n$, $\SuCom(n)$ is a singleton $\{*_n\}$ equipped with a trivial action of $\∑_n$, with unit $*_1\in\SuCom(1)$ and with composition given by identities of singletons. $\SuCom$-algebras are exactly unital, commutative, associative monoids.
            \item If $\mathfrak A$ is a unital, associative monoid, then $\mathfrak A$ can be considered as an operad concentrated in arity $1$. $\mathfrak A$-algebras correspond to sets with a (left) $\mathfrak A$-action.
        \end{itemize}
        Extending the examples of operads in set by linearity induce examples of operads in $\F$-vector spaces. Here are some basic examples of operads in $\F$-vector spaces:
        \begin{itemize}
            \item Denote by $\uCom=\F \SuCom$, the linearisation of $\SuCom$. One has $\uCom(n)=\F$ for all $n\in\N$, the unit of this operad is $1\in\F=\uCom(1)$, compositions are given by identities of $\F$, and $\uCom$-algebras are unital, associative, commutative $\F$-algebras. Let $X_n\in\uCom(n)$ denote the generator of the vector space $\uCom(n)$.
            \item If $A$ is a unital, associative $\F$-algebra, then $A$ can be considered as an operad concentrated in arity $1$. $A$-algebras correspond to the classical notion of $A$-(left)-modules.
        \end{itemize}
    \end{ex}
    We now introduce the notion of a distributive law for product of operads. Note that distributive laws are usually defined in a different yet equivalent way {\cite[Section 8.6.1]{LV}}.
\begin{defi}[Distributive Law]\label{def:distrlaw}
    Let $\P$, $\Q$ be two operads. A \textbf{distributive law} is a morphism $\lambda:\Q\circ\P\to \P\circ \Q$ in $\Sym$ such that the symmetric sequence $\P\circ \Q$, with unit $(1_\P;1_\Q)$, and with composition map given by:
    \[
    \diag{\P\circ\Q\circ\P\circ\Q\ar[rr]^-{\P\circ\lambda\circ\Q}&&\P\circ\P\circ\Q\circ\Q\ar[rr]^-{\gamma_P\circ\gamma_Q}&&\P\circ \Q},
    \]
    is an operad.
\end{defi}
    
    \section{Modules over a bialgebra}\label{sec:modbialg}
    In this section, we study the category of (left) modules over a bialgebra. This category is classically equipped with a tensor product which is symmetric as soon as the bialgebra is cocommutative. This will allow us to study operads in this category, and their algebras. Our main reference for the notion of bialgebras and their modules is the textbook account \cite{CKa}.
    
    The primary purpose of this article is to study unstable modules over the Steenrod algebra. In \cite{JM}, Milnor shows that this algebra is an example of cocommutative bialgebra. This will allow us, in Section \ref{sec:unstalg}, to study algebras over the Steenrod algebra using the notion of operads.

    \begin{defi}[Bialgebra {\cite[Section III.2]{CKa}}]
         A \textbf{bialgebra} is a quintuple $(B,\mu,\eta,\Delta,\epsilon)$ where $(B,\mu,\eta)$ is a unital associative $\F$-algebra, $(B,\Delta,\epsilon)$ is a counital, coassociative $\F$-coalgebra, such that $\epsilon$ and $\Delta$ satisfy compatibility conditions that make them into algebra morphisms.

        The category of (left) modules over a bialgebra $B$ will be denoted by $B_{\mod}$.
    \end{defi}
    \begin{defi}[Tensor product of $B$-modules {\cite[III.5]{CKa}}]\label{def:tensormod}
        Let $M$ and $N$ be two $B$-modules. We define a $B$-module structure on $M\otimes N$ by:
        \[
        \diag{B\otimes M\otimes N \ar[rr]^-{\Delta\otimes M\otimes N}&&B\otimes B\otimes M\otimes N\ar[rr]^-{B\otimes \tau\otimes N}&&B\otimes M\otimes B\otimes N\ar[r]&M\otimes N},
        \]
        where $\tau:B\otimes M\to M\otimes B$ is the symmetry isomorphism of the tensor product in $\Vect$ and the last map is given by the $B$-module structures on $M$ and $N$.

        This provides a tensor product that we still denote by $\otimes$ on $B$-modules.
    \end{defi}
    \begin{lemm}[{\cite[Proposition 111.5.1]{CKa}}]
        If $B$ is cocommutative, then the tensor product of Definition \ref{def:tensormod} is symmetric.
    \end{lemm}
    \begin{defi}[$\P$-algebras in $B$-modules]\label{def:PalginBmod}
        Let $B$ be a cocommutative bialgebra, and $\P$ be an operad in $\F$-vector spaces. Consider $\P$ as an operad in the symmetric monoidal category $B_{\mod}$ by declaring that $B$ acts trivially on $\P$. A \textbf{$\P$-algebra in $B$-modules} is an algebra over the resulting operad in $B_{\mod}$

        $\P$-algebras in $B$-modules form a category denoted by $\P_{\alg}^B$.
    \end{defi}
    \begin{rema}\label{rema:distr.law}
        $\P$-algebras in $B$-modules can also be defined as $\P\circ B$ algebras with a certain distributive law $\lambda:B\circ \P\to \P\circ B$.
        
        Denote by $\Delta^{n-1}=(B^{\otimes n-2}\otimes\Delta)\circ\dots\circ (B\otimes \Delta)\circ\Delta:B\to B^{\otimes n}$ the $n-1$-th iterated coproduct. In arity $n$, the map $\lambda:(B\circ \P)(n)\to(\P\circ B)(n)$ is the composite:
        \[
            \lambda:\diag{(B\circ \P)(n)=B\otimes \P(n)\ar[r]&\P(n)\otimes B\ar[r]^{\P\otimes \Delta^{n-1}}&\P(n)\otimes B^{\otimes n}\ar[r]&\P(n)\otimes_{\∑_n}B^{\otimes n}=(\P \circ B)(n)},
        \]
    \end{rema}
    where the first arrow is the braiding of the tensor product, and the third map is the projection on orbits by the action of $\∑_n$.

    We now present a basic example of bialgebra and identify the associated modules and $\P$-algebras. This example and all the variations introduced here are fundamental for the rest of this article. In particular, the bialgebras $T_s\D$, $Q_s\D$ and $\Dp$ will appear in Section \ref{sec:BGC} to compare certain unstable algebras to known unstable modules. A more involved example of bialgebra is given by the Steenrod algebra. This example is discussed in Section \ref{sec:Steenrodalg}.
    \begin{ex}\label{ex:bialgebras}
    Let $\D=\F[d]$ denote the polynomial algebra in one indeterminate $d$. Following \cite[III.2.Example 2]{CKa}, we can equip $\D$ with a bialgebra structure such that $\Delta(z d^n)=z(d^n\otimes d^n)$ for all $n\in \N$, $z\in\F$. The $\D$-modules are vector spaces $V$ equipped with a linear map $d:V\to V$. The $\P$-algebras in $\D$-modules are the $\P$-algebras $A$ equipped with a $\P$-algebra morphism $d:A\to A$.

            We can construct many variations of the previous example. For instance:
        \begin{itemize}
            \item Let $T_sD$ denote the quotient of $\D$ by the ideal generated by $d^{s+1}$. The bialgebra structure of $\D$ induces a bialgebra structure on $T_s\D$. $T_s\D$-modules (resp. $\P$-algebras in $T_sD$-modules) are $\D$-modules (resp. $\P$-algebras in $\D$-modules) such that $d$ is nilpotent of degree $s+1$.
            \item Let $Q_sD$ denote the quotient of $\D$ by the ideal generated by $d^{s}-1$. The bialgebra structure of $\D$ induces a bialgebra structure on $Q_s\D$. $Q_s\D$-modules (resp. $\P$-algebras in $Q_sD$-modules) are $\D$-modules (resp. $\P$-algebras in $\D$-modules) such that $d$ is cyclic of degree $s+1$.
            \item Let $\Dp=\D[d,d^{-1}]$, the algebra of Laurent polynomials in one indeterminate. The bialgebra structure of $\D$ extends into a bialgebra structure on $\Dp$. $\Dp$-modules (resp. $\P$-algebras in $\Dp$-modules) are $\D$-modules (resp. $\P$-algebras in $\D$-modules) such that $d$ is an isomorphism.
        \end{itemize}
    \end{ex}
    \section{Operad of partitions}\label{sec:levpart}
    In this section, we introduce and study a particular example of operad, the operad of (ordered) partitions $\Pi$. We show that this operad is isomorphic to a composition product of two operads equipped with a distributive law. This operad $\Pi$ will be crucial in order to define the operad of $q$-level algebras in Section \ref{sec:Lev}, and will be used in Section \ref{sec:BGC} to identify certain unstable algebras over the Steenrod algebra.
    \begin{defi}
        An \textbf{ordered partition} $J$ of the set $[n]=\{1,\dots,n\}$ is a sequence $J=(J_i)_{i\in\N}$ of piecewise disjoint subsets of $[n]$ such that $\bigcup_{i\in\N}J_i=[n]$.

        Since $[n]$ is finite, there is an integer $s$ such that $J_s\neq\emptyset$ and $J_{s'}=\emptyset$ for all $s'>s$. We will often write $J=(J_0,\dots,J_s)$, omitting the empty sets that follow.

        Let $j,k,l$ be three positive natural number such that $l\le j$. Denote by $\lambda_{l,k}^j:[j]\to[j-1+k]$ the set map defined by
        \[
            \lambda_{l,k}^j(m)=\begin{cases}
            m,&\mbox{if }m\le l,\\
            m-1+k &\mbox{if }m>l.
            \end{cases}
        \]
        In other words, $\lambda_{l,k}^j$ ‘‘skips'' all numbers between $l+1$ and $l-1+k$.
    \end{defi}
    \begin{defi}[The operad of partitions $\Pi$]
        There is an operad $\Pi$ in sets, called the operad of partitions, such that, for all $n\in\N$, $\Pi(n)$ is the set of ordered partitions $J=(J_i)_{0\le i\le s}$ of the set $[n]$. The unit of the operad is the partition $(\{1\})\in\Pi(1)$ of $[1]$, and the partial compositions are induced by:
  \[(J\circ_l K)_i=\begin{cases}
      \lambda_{l,k}^j(J_i),&\mbox{if }l\in J_{i'}\mbox{ with }i'>i,\\
      \lambda_{l,k}^j(J_i)\setminus\{l\}\cup (K_{0}+l-1)&\mbox{if }l\in J_{i},\\
      \lambda_{l,k}^j(J_i)\cup (K_{i-i'}+l-1),&\mbox{if }l\in J_{i'}\mbox{ with }i'<i,
    \end{cases}\]
  where $K$ is a partition of $[k]$, $J$ is a partition of $[j]$, $l\in [n]$ and for a subset $S$ of $\N$ and an integer $m\in\N$, let $S+m$ be the set $\{x+m\colon x\in S\}$.
    \end{defi}
    We give two example of composition of ordered partitions. The second example will be of key importance later in this section.
    \begin{ex} 
        Consider $J=(\{2\},\{1,3\})\in \Pi(3)$ and $K=(\emptyset,\{1,2\})\in\Pi(2)$. Let's compute $J\circ_1 K$. For this, notice that $\lambda_{1,2}^3:[3]\to[4]$ sends $1$ to $1$, $2$ to $3$ and $3$ to $4$. So, $\lambda_{1,2}^3(J_0)=\{3\}$, $\lambda_{1,2}^3(J_1)=\{1,4\}$, $K_0=K_0+1-1=\emptyset$, $\lambda_{1,2}^3(J_2)=\emptyset$, and $K_{1}+l-1=K_1=\{1,2\}$, so finally,
        \[
    J\circ_1K=(\{3\},\{4\},\{1,2\}).
        \]
    \end{ex}
    \begin{ex}\label{ex:comppart}
        Consider $J=(\emptyset,[q])\in\Pi(q)$ for a certain integer $q\ge 2$. We want to compute $(J\circ_2J)\circ_1J$. Note first that $\lambda_{2,q}^q$ sends $1$ to $1$, $2$ to $2$, and all $n\in\{3,\dots,q\}$ to $n+q-1$. So, $\lambda_{2,q}^q(J_0)=\emptyset$, $\lambda_{2,q}^q(J_1)=\{1,2,q+2,q+3,\dots,2q-1\}$, $J_0+2-1=\emptyset$, and $J_1+2-1=J_1+1=\{2,3,\dots,q+1\}$. So,
        \[
        J\circ_2J=(\emptyset,\{1,q+2,q+3,\dots,2q-1\},\{2,3,\dots,q+1\}).
        \]
        Now, note that $\lambda_{1,q}^{2q-1}:[2q-1]\to[3q-2]$ send $1$ to $1$ and all $n\in\{2,3,\dots,2q-1\}$ to $n+q-1$. So, $\lambda_{1,q}^{2q-1}((J\circ_2J)_0)=\emptyset$, $\lambda_{1,q}^{2q-1}((J\circ_2J)_1)=\{1,2q+1,2q+2,\dots,3q-2\}$, $\lambda_{1,q}^{2q-1}((J\circ_2J)_2)=\{q+1,q+2\dots,2q\}$. Since $J_0+1-1=\emptyset$ and $J_1+1-1=J_1=\{1,\dots,q)$, we conclude:
        \[ 
        (J\circ_2J)\circ_1J=(\emptyset,\{2q+1,2q+2,\dots,3q-2\},\{1,2,\dots,2q\}).
        \]
    \end{ex}
    The operad $\Pi$ can be identified to a product of operads with distributive law (see Definition \ref{def:distrlaw}). Recall from Example \ref{ex:operads} that there is an operad $\SuCom$ in sets whose algebras are unital, associative, commutative monoids. Denote by $\SD$ the unital, associative monoid whose underlying set is $\{d^i\}_{i\in\N}$ with multiplication $d^i\cdot d^j=d^{i+j}$. According to Example \ref{ex:operads}, $\SD$ can be seen as an operad in $\Set$ concentrated in arity 1. We prove the following:
    \begin{prop}\label{propisopi}
  There is an isomorphism of operads:
  \[\SuCom\circ\SD\to\Pi,\]
  where $\SuCom\circ\SD$ is endowed with the operad structure coming from the distributive law $\lambda:\SD\circ\SuCom\to\SuCom\circ\SD$ such that:
  \[
    \lambda(d^j;*_n)=(*_n,\underbrace{d^j,\dots,d^j}_n).
  \]
\end{prop}
\begin{proof}
  Verifying that $\lambda$ above satisfies the definition of a distributive law is straightforward, and left to the reader.
  
  Note that $\SuCom$ is generated as an operad by $*_2$, and $\SD$ is generated by $d^1$. This implies that $\SuCom\circ\SD$ is generated by $(*_2;d^0,d^0)$ and $(*_1;d^1)$. Let us define a morphism of operads $\varphi\colon\SuCom\circ \SD\to \Pi$, induced by 
    \begin{equation*}
      \begin{array}{lr}
      \varphi(*_2;d^0,d^0)=(\{1,2\})\in\Pi(2),& \varphi(*_1;d^1)=(\emptyset,\{1\})\in\Pi(1).
    \end{array}
    \end{equation*}
  One can express $\varphi$ on any element of the form $\t=(*_n;d^{i_1},\dots,d^{i_{n}})\in\SuCom\circ\SD(n)$: the partition $\varphi(\t)$ of $[n]$ is such that, for all $j\in[n]$,~$j\in\varphi(\t)_{i_j}$. It is easy to check that this uniquely defines $\varphi(\t)$ and that we have described a bijection between a basis of $\uCom\circ\D$ and a basis of $\Pi$ compatible with composition.
\end{proof}
From this proposition we obtain the straightforward corollary:
\begin{coro}
    A $\Pi$-algebra is a commutative monoid equipped with a commutative monoid endomorphism.
\end{coro}
\begin{rema}\label{rema:uComcircD}
    We can transfer the preceding construction to the category of $\F$-vector spaces using linearisation. Recall from Example \ref{ex:operads} that the linearisation of $\SuCom$ is the operad $\uCom$ of unital, associative, commutative $\F$-algebras. Note that the linearisation of $\SD$ is $\D$, the polynomial algebra in one indeterminate $d$. The previous result yields an isomorphism of operads between $\uCom\circ\D$, equipped with the distributive law of \ref{rema:distr.law} (with $B=D$ seen as a bialgebra using Example \ref{ex:bialgebras}), and the linearisation $\F\Pi$ of $\Pi$.
\end{rema}
    \section{Operad of \texorpdfstring{$q$}{q}-level algebras}\label{sec:Lev}
    In this section, we introduce the operad $\Lev_q$ of $q$-level algebras. This is a new construction which generalises the operad $\Lev$ of level algebras defined by Chataur and Livernet \cite{CL}. Intuitively, a $q$-level algebra is a vector space endowed with a $q$-ary operation which is strongly commutative in a certain sense, but which is not associative. This operad will be used to identify certain unstable modules over the Steenrod algebra, called the Carlsson modules, in section \ref{sec:BGC}.
    \begin{defi}[The operads $\SLev_q$, $\Lev_q$]\label{defi:Lev}
        For all integer $q>1$, \textbf{the operad of $q$-level monoids} is the sub-operad $\SLev_q\subseteq \Pi$ generated by the element $(\emptyset,[q])\in\Pi(q)$.

        For all integer $q>1$, \textbf{the operad of $q$-level algebras} is the linearisation $\Lev_q=\F\SLev_q$ of the operad of $q$-level monoids.

        As suggested in the name of this operad, we will call \textbf{$q$-level algebras} the algebras over the operad $\Lev_q$.
    \end{defi}
    Recall that operads can be described by generators and relations \cite[5.4.5]{LV}. We now present one of our main new results, which gives a presentation for the operad $\SLev_q$, and by linearisation, for the operad $\Lev_q$:
    \begin{theo}\label{theo:Lev}
        The operad $\SLev_q$ is generated by one element $\star\in\SLev_q(q)$, fixed under the action of $\∑_q$, and satisfying the unique following relation:

    \[(\star\circ_2\star)\circ_1\star=(\star\circ_2\star)\circ_1\star \cdot \sigma,\]
    where $\sigma\in\∑_{3q-2}$ is the transposition $(q\  q+1)$.
    \end{theo}
    The proof of this result, which is somewhat long and technical, is postponed to Section \ref{sec:proofLev}.

    Let's illustrate the notion of $q$-level algebra through examples:
    \begin{ex}
        A $3$-level algebra is a vector space $A$ equipped with a ternary operation $\star$ such that, for all $(a_i)_{i\in[7]}$,
        \[
        \star(\star(a_1,a_2,a_3),\star(a_4,a_5,a_6),a_7)=\star(\star(a_1,a_2,a_4),\star(a_3,a_5,a_6),a_7).
        \]
        We can represent any composition of $\star$ by corollas with $3$ inputs. For example, the above equality can be represented by:
        \[
    \begin{tikzpicture}[grow'=up,level distance=25pt]
      \Tree 
      [.$\star$  [.$\star$ $a_1$  $a_2$ $a_3$ ] [.$\star$  $a_4$ $a_5$ $a_6$ ] $a_7$ ]
    \end{tikzpicture}=\begin{tikzpicture}[grow'=up,level distance=25pt]
      \Tree 
      [.$\star$  [.$\star$ $a_1$  $a_2$ $a_4$ ] [.$\star$  $a_3$ $a_5$ $a_6$ ] $a_7$ ]
    \end{tikzpicture}.
  \]
  Using the fact that $\SLev_q$ is a suboperad of $\Pi$, one can see that the element $a_i$ will be at height $j$ in this composition tree if and only if the index $i$ was in the $j$-th component of the partition given by $(\star\circ_2\star)\circ_1\star\in\Pi(7)$. This is true for all compositions of $\star$ seen as composition trees. Using the relations within $\Pi$, this means that two elements in a product can be interchanged whenever they are at the same height in the composition tree. For example, the product $\star(\star(\star(a_1,a_3,a_5),a_6,a_8),a_2,\star(a_4,a_9,a_7))$ has composition tree:
  \[
\begin{tikzpicture}[grow'=up,level distance=25pt]
      \Tree 
      [.$\star$  [.$\star$ [.$\star$ $a_1$  $a_3$ $a_5$ ] $a_6$ $a_8$ ] $a_2$ [.$\star$  $a_4$ $a_9$ $a_7$ ] ]
    \end{tikzpicture},
  \]
  and this means that in the $3$-level algebra $A$, the resulting product does not change if applying any permutation of $a_1,a_3,a_5$, or of $a_6,a_8,a_4,a_9,a_7$ (but $a_2$ cannot, \textit{a priori}, be transposed with any other element in the product). 
    \end{ex}
\begin{rema}
    The operad $\Lev_2$ is the operad $\Lev$ of level algebras of \cite{CL}.
\end{rema}
    We now construct a cofiltration of $\Lev_q$ using a notion of truncation. This new family of operads will be used in Section \ref{sec:BGC} to identify certain unstable modules.

    For all $s>0$, $n\in\N$, denote by $\SLev_q^{>s}(n)$ the subset of $\SLev_q(n)$ of partitions $J$ such that there exists $i>s$ such that $J_i\neq\emptyset$. It is easy to check that $\SLev_q^{>s}\subset\SLev_q$ forms an operadic ideal (see \cite[5.2.16]{LV}), which allows us to define an operad structure on the quotient $\SLev_q/\SLev_q^{>s}$.
    \begin{defi}\label{def:TsLev}
        The \textbf{$s$-truncation} of $\SLev$ is the quotient operad $T_s\SLev_q=\SLev_q/\SLev_q^{>s}$.

        The \textbf{$s$-truncation} of $\Lev$ is the quotient operad $T_s\Lev_q=\Lev_q/\Lev_q^{>s}$, where $\Lev_q^{>s}=\F\SLev_q^{>s}$.
    \end{defi}
    \begin{rema}\label{rema:cof}
         Since we have a filtration $\SLev_q^{>1}\supset\SLev_q^{>2}\supset \cdots$ whose limit is $\emptyset$, we obtain a cofiltration of $\Lev_q$:
    \begin{equation}\tag{$*$}\label{diag:cof}
        \diag{\Lev_q\ar[r]&\dots\ar[r]&T_{s+1}\Lev_q\ar[r]&T_{s}\Lev_q\ar[r]&\dots}
    \end{equation}
    \end{rema}
    \section{Central operations}\label{sec:centop}
    In this section, we introduce the notion of \textit{central operation} in an operad. These operations satisfy a certain interchange relation with respect to all other operations in the operad. We give some motivation for this definition: more precisely, the condition of centrality for a $q$-ary operation $\star$ in an operad $\P$ is the minimal condition required to define an algebra morphism $x\mapsto \star(x,\dots,x)$ in any algebra over $\P$. This notion of $\P$-centrality generalises the notion of $\P$-centrality for binary commutative operations defined in \cite[Definition 5.2]{SI2}.

    Given a $\P$-central operation $\star$, we construct a functor $\psi_!$ which takes a vector space $V$ with a linear self-map $d$ and freely produces a $\P$-algebra over $V$ satisfying $dv=\star(v,\dots,v)$. This construction allows us to produce free unstable algebras over an unstable module in Section \ref{sec:unstalg}. The results of the present section, in particular the commutative diagram \ref{diag:freecomp} of \ref{coro:freecomp}, are crucial to the proof of Theorem \ref{theo:freeunst} on the structure of free unstable algebras.
    
    Recall that $\F$ is a field of characteristic $p$ and order $q=p^\alpha$, and that $D=\F[d]$ is the polynomial algebra in one indeterminate $d$ (see Example \ref{ex:bialgebras}).

    Fix an operad $\P$ and an arity $q$ operation $\star\in\P(q)$.
    \begin{nota}
       Let $\star_n\in\P(q^n)$ denote the operation inductively defined by:
        \[
        \star_0=1_\P\quad\mbox{and}\quad \star_{n+1}=\star(\underbrace{\star_n,\dots,\star_n}_{q}).
        \]
    \end{nota}
    In particular, $\star_1=\star$. The associativity of the composition in $\P$ implies that:
    \[\star_i(\underbrace{\star_j,\dots,\star_j}_{q^i})=\star_j(\underbrace{\star_i,\dots,\star_i}_{q^j})=\star_{i+j}.\]

    For all vector spaces $V$, note that, since the base field is of order $q$, the following induces a linear map:
    \begin{eqnarray*}
        \psi_V:&D\otimes V&\to S(\P,V)\\
        &d^n\otimes v&\mapsto v^{\star_nq^n},
    \end{eqnarray*}
    where the notation $a^{\mu n}$ was introduced in Definition \ref{defalgop}. This map is clearly natural in $V$, and induces a monad morphism $\psi:D\otimes-\to S(\P,-)$. This induces a functor $\psi^*:\P_{\alg}\to \D_{\mod}$ on the categories of algebras which restricts the structure along $\psi$ (see \cite[Section 3.6]{BW}). This functor admits a left adjoint that can be constructed for example as a left Kan extension (see \cite[Chapter X]{SML}).
    \begin{nota}\label{nota:psi!}
        Denote by $\psi_!:\D_{\mod}\to \P_{\alg}$ a left adjoint to the functor $\psi^*$.
    \end{nota}
    Explicitely, for a $\D$-module $M$, the $\P$-algebra $\psi_!(M)$ is obtained as a coequaliser between two $\P$-algebra morphism $S(\P,\D\otimes M)\to S(\P,M)$. The first of these $\P$-algebra morphism is given by the $\D$-action $\D\otimes M\to M$. The second is given by the composition:
    \[ 
        \diag{S(\P,\D\otimes M)\ar[rr]^-{S(\P,\psi_M)}&&S(\P,S(\P,M))\ar[r]^-{(\gamma_\P)_M}&S(\P,M)},
    \]
    where the first map is $\psi$ applied to the underlying vector space of $M$, and the second map is the composition of the monad $S(\P,-)$. In other words, $\psi_!(M)$ is the quotient of the free $\P$-algebra $S(\P,M)$ by the $\P$-ideal generated by the elements $d^nm-m^{\star_n q^n}$ for $m\in M$ and $n\in\N$ (see \cite[2.3.1]{GK} for the notion of ideal in a $\P$-algebra). 
    
    For a $\D$-module $M$, the resulting $\P$-algebra $\psi_!(M)$ can be equipped again with a $\D$-module structure through $\psi^*$, that is, by declaring that, for all $\mu\in\P(n)$, $m_1,\dots,m_n\in M$,
    \[d(\mu;m_1,\dots,m_n)=\star(\underbrace{(\mu;m_1,\dots,m_n),\dots,(\mu;m_1,\dots,m_n)}_q).\]
    So, $\psi_!(M)$ is a $\P$-algebra and a $\D$-module, but it is not necessarily a $\P$-algebra in $\D$-modules as defined in \ref{def:PalginBmod}. To extend $\psi_!$ into a functor $\D_{\mod}\to{\P}_{\alg}^{\D}$, we need $\star$ to satisfy an additional condition, which we call $\P$-centrality:
    \begin{defi}[$\P$-central operations]\label{def:centop}
        An operation $\star\in\P(q)$ is said to be \textbf{$\P$-central} if, for all $\mu\in\P(n)$,
        \begin{equation}
    \star(\mu,\dots,\mu)=\mu(\star,\dots,\star)\cdot \sigma_{q,n},\tag{I}\label{interchange}
  \end{equation}
  where $\sigma_{q,n}\in\∑_{qn}$ sends $(i-1)n+k$ to $(k-1)q+i$ for all $i\in[q]$,~$k\in[n]$. Under the usual identifications $[qn]\cong [q]\times[n]$ and $[qn]\cong [n]\times [q]$ the permutation $\sigma_{q,n}$ corresponds to the transposition $[q]\times[n]\to [n]\times [q]$, $(i,k)\mapsto(k,i)$.
    \end{defi}
    We get the straightforward result:
    \begin{lemm}\label{lemm:psiext}
        Let $\star\in\P(q)$. The functor $\psi_!:\D_{\mod}\to\P_{\alg}$ defined in \ref{nota:psi!} can be extended into a functor $\psi_!:\D_{\mod}\to\P_{\alg}^{\D}$ if and only if $\star$ is $\P$-central. In other words, when $\star$ is $\P$-central, there is a commutative diagram of categories:
        \[
        \diag{&\P_{\alg}\\\D_{\mod}\ar[ur]^{\psi_!}\ar[r]^{\psi_!}&\P_{\alg}^{\D}\ar[u]_-{U}},
        \]
        were the $U$ is the forgetful functor that extracts the underlying $\P$-algebra.
    \end{lemm}
    \begin{proof}
        For any $\D$-module $M$, consider the action of $D$ on $\psi_!(M)$ obtained through $\psi^*$ (this action is described above). The condition for the action of $d$ on $\psi_!(M)$ to induce a $\P$-algebra morphism is precisely the condition of $\P$-centrality for $\star$.
    \end{proof}
    Before illustrating the notion of $\P$-centrality with some examples, we now prove the main result of this section, which will allow us to prove Theorem \ref{theo:freeunst}. 
\begin{prop}\label{prop:freecomp}
    Let $\star\in \P(q)$. There is a natural isomorphism:
    \[
    \psi_!\circ(D\otimes -)\cong S(\P,-).
    \]
\end{prop}
\begin{proof}
    Consider the following diagram of categories and adjunctions:
    \[
    \diag{\Vect\ar@<0.5ex>[rr]^{S(\P,-)}\ar@<0.5ex>[dd]^{D\otimes -}&&\P_{\alg}\ar@<0.5ex>[ll]^{\Forget}\ar@<0.5ex>[lldd]^{\psi^*}\\\\
    \D_{\mod}\ar@<0.5ex>[uu]^{\Forget}\ar@<0.5ex>[uurr]^{\psi_!}},
    \]
    where $\Forget$ denotes all the right adjoints that extract underlying vector spaces. One can easily check that the diagram formed by right adjoints commute. Since $\psi_!\circ(\D\otimes-)$ is a left adjoint of $\Forget\circ \psi^*$, it is also a left adjoint of $\Forget:\P_{\alg}\to \Vect$. Left adjoints being unique up to natural isomorphisms, we obtain the desired natural isomorphism $\psi_!\circ(D\otimes -)\cong S(\P,-)$.
\end{proof}
Combining Lemma \ref{lemm:psiext} and Proposition \ref{prop:freecomp}, one gets the following:
\begin{coro}\label{coro:freecomp}
    Let $\star\in \P(q)$ be a $\P$-central operation. The following diagram of categories commutes:
    \begin{equation}\tag{D1}\label{diag:freecomp}
         \diag{\Vect\ar[d]_{\D\otimes -}\ar[r]^{S(\P,-)}&\P_{\alg}\\\D_{\mod}\ar[r]^{\psi_!}&\P_{\alg}^{\D}\ar[u]_-{U}},
    \end{equation}
\end{coro}
%     \begin{ex}
%   Let $\star\in\P(3)$ be a $\P$-central operation. Let $\mu\in\P(2)$, and $A$ be a $\P$-algebra. Then for all $a,b,c,d,e,f\in A$, one has:
%   \[
%     \star(\mu(a,b),\mu(c,d),\mu(e,f))=\mu(\star(a,c,e),\star(b,d,f)),   
%   \]
%   This equality can be represented with the following tree equality:
%   \[
%     \begin{tikzpicture}[grow'=up,level distance=25pt]
%       \Tree 
%       [.$\star$  [.$\mu$ $a$  $b$ ] [.$\mu$  $c$ $d$ ] [.$\mu$ $e$ $f$ ] ]
%     \end{tikzpicture}=\begin{tikzpicture}[grow'=up,level distance=25pt]
%       \Tree 
%       [.$\mu$  [.$\star$ $a$  $c$ $e$ ] [.$\star$  $b$ $d$ $f$ ] ]
%     \end{tikzpicture}.
%   \]
% \end{ex}

The next proposition allows one to check, from a presentation of an operad $\P$, if an operation $\star\in\P(q)$ is $\P$-central. It is used in the example~\ref{interchangeop} to produce objects that will appear in Sections \ref{sec:BGC}.
\begin{prop}\label{interchangegen}
   Let~$\star\in\P(q)$. Let $F$ be a sub-symmetric sequence of~$\P$. Suppose that $F$ generates the operad~$\P$. The operation $\star$ is $\P$-central if and only if it satisfies relation \eqref{interchange} of Definition \ref{def:centop} for all~$\mu\in F$.
\end{prop}
\begin{proof}
  This is a fairly straightforward generalisation of \cite[Proposition 5.7]{SI2}.
  % It suffices to check that if $\mu,\nu\in F$ satisfy (\eqref{interchange}), then all partial compositions $\mu\circ_i\nu$ satisfy (\eqref{interchange}), and, in this setting, if $\mu$ and $\nu$ have same arity, $\mu+\nu$ satisfies (\eqref{interchange}). The additivity statement is easily checked. Let~$\mu,\nu\in F$. Denote by $m$ and $n$ the respective arities of $\mu$ and of~$\nu$. If $\mu$ and $\nu$ satisfy (\eqref{interchange}), then:
  % \begin{align*}
  %   (\mu\circ_i\nu)^{\star q}&=\left(\left(\dots\left(\mu^{\star q}\right)\circ_{(q-1)m+i}\nu\right)\circ_{(q-2)m+i}\nu\dots\right)\circ_i\nu,\\
  %   &=\left(\left(\dots(\mu(\star^{\times m}) \sigma_{q,m}\circ_{(q-1)m+i}\nu\right)\circ_{(q-2)m+i}\nu\dots\right)\circ_i\nu,\\
  %   &=\left(\left(\mu(\star^{\times m})\circ_{iq}\nu)\circ_{iq-1}\nu\dots\right)\circ_{(i-1)q+1}\nu\right) \sigma',\\
  %   &=\left(\mu\left(\star,\dots,\star,\underbrace{\nu^{\star q}}_i,\star,\dots,\star\right)\right) \sigma',\\
  %   &=\left(\mu\left(\star,\dots,\star,\underbrace{\nu(\star^{\times n}) \sigma_{q,n}}_i,\star,\dots,\star\right)\right) \sigma',\\
  %   &=(\mu\circ_i\nu)(\star,\dots,\star) \sigma_{q,(m+n-1)},
  % \end{align*}
  % where $\sigma'\in\∑_{qm+qn-q}$ is the block permutation obtained by applying $\sigma_{q,n}$ to $qm$ blocks of size $1$, except for the $i$-th, $(i+1)$-th,\dots,$(i+q-1)$-th blocks, of size~$n$.
\end{proof}
Let us now give some examples of operads endowed with central operations.
\begin{ex}\label{interchangeop}
  The generator of the vector space $\uCom(q)$ is $\uCom$-central. The generator $\star\in\Lev_q(q)$ of the operad $\Lev_q$ defined in Definition \ref{defi:Lev} is $\Lev_q$-central. 

  Denote by $T_s\Lev_q$ the image of the composite $\Lev_q\hookrightarrow \Pi\cong \uCom\circ\D\to\uCom\circ T_s\D$. The operadic generator $\star\in\Lev_q(q)$ yields a $T_s\Lev_q$-central operation (one can show this using Proposition~\ref{interchangegen}). Similarly, the operation $(X_q;d^{\times q})$, where $X_q$ is the generator of the arity module $\uCom(q)$, is a $\uCom\circ\D$-central operation, a $\uCom\circ\Dp$-central operation, and a $\uCom\circ Q_s\D$-central operation. More generally, if $\star\in\P(q)$ is $\P$-central, then $(d^i)^{\star q}$ is $\P\circ\D$, $\P\circ\Dp$, and $\P\circ Q_s\D$-central, for all~$i\in\N$.
\end{ex}
\section{Proof of Theorem \ref{theo:Lev}}\label{sec:proofLev}
This section is devoted to the proof of Theorem \ref{theo:Lev}, which gives a presentation of the operad $\Lev_q$ of Definition \ref{defi:Lev}, and gives us insights on its associated algebras. This proof uses the description of free operads using trees (see, for example, \cite{LV}, section 5.4). Before proving this theorem, we need one preliminary result:
\begin{lemm}\label{lemm:partLev}
  As a subset of $\Pi(n)$, $\SLev_q(n)$ is the set of all ordered partitions $J=(J_i)_{0\le i\le s}$ of $[n]$ such that:
  \begin{equation}\tag{$**$}\label{cond}
     \sum_{i=0}^s\frac{|J_i|}{q^i}=1.
  \end{equation}
\end{lemm}
\begin{proof}
  By induction on $n$. Since $\SLev_q(n)$ is generated by a single element $(\emptyset,[p])$ of arity $q$, $\SLev_q(0)=\emptyset$, $\SLev_q(1)$ is equal to $\{1_\Pi\}=\{([1])\}$, $\SLev_q(q)$ is equal to $\{(\emptyset,[p])\}$, and $\SLev_q(i)=\emptyset$ for all $1<i<q$. It is easy to see that $([1])$ is the only partition of $[1]$ satisfying equation \eqref{cond}, that $(\emptyset,[p])$ is the only partition of $[p]$ satisfying equation \eqref{cond}, and that no partition of $\emptyset=[0]$, and no partition of $[2],[3],\dots,[q-1]$ can satisfy equation \eqref{cond}.

  Suppose now that $n>q$ and that we have proven that $\SLev_q(n-q+1)$ is the set of all ordered partitions $J=(J_i)_{0\le i\le s}$ of $[n-q+1]$ satisfying equation \eqref{cond}.

  On one hand, let $J=(J_i)_{0\le i\le s}$ a partition of $[n]$, and suppose that $J\in \SLev_q(n)$. Then, since $n>q$ and $\SLev_q$ is spanned by $(\emptyset,[q])$, there exist a partition $\hat J=(\hat J_i)_{0\le i\le s}$ of $[n-q+1]$, an integer $k\in[n-q+1]$, and a permutation $\sigma\in\∑_n$ such that $J=\hat J\circ_k(\emptyset,[q])\cdot \sigma$. By the induction hypothesis, $\hat J$ satisfies equation \eqref{cond}. Denote by $i_0$ the unique integer such that $k\in \hat J_{i_0}$. Then,
  \[
    \sum_{i=0}^s\frac{|J_i|}{q^i}-\sum_{i=0}^s\frac{|\hat J_i|}{q^i}=\frac{1}{q^{i_0}}-\frac{q}{q^{i_0+1}}=0,
  \]
  so $J$ satisfies equation \eqref{cond}.

  On the other hand, let $J=(J_i)_{0\le i\le s}$ be a partition of $[n]$, and suppose that $J$ satisfies equation \eqref{cond}. One can assume that $J_s$ is non-empty without loss of generality. Then, since $\sum_{i=0}^s\frac{|J_i|}{q^i}=1$ and $n>q$, this implies that $s>1$, and
  \[
    |J_s|=q^s\left(1-\sum_{i=0}^{s-1}\frac{|J_i|}{q^i}\right)=q\left(q^{s-1}-\sum_{i=0}^{s-1}|J_i|q^{s-1-i}\right)
  \]
  So $|J_s|$ is a multiple of $q$. In particular, since $J_s\neq\emptyset$, $|J_s|\ge q$. There exists a permutation $\sigma\in\∑_n$ such that $\{n-q+1,\dots,n\}\in(J\cdot\sigma)_s$. Let $\bar J\in\Pi(n-q+1)$ be the partition of $[n-q+1]$ such that
  $\bar J_i=(J\cdot\sigma)_i$ for all $i<s-1$, $\bar J_{s-1}=\{n-q+1\}\cup (J\cdot\sigma)_{s-1}$, and $\bar J_s=\emptyset$. Then, one has:
  \[
    \sum_{i=0}^s\frac{|\bar J_i|}{q^i}=\left(\sum_{i=0}^s\frac{|J_i|}{q^i}\right)-\frac{q}{q^s}+\frac{1}{q^{s-1}}=1,
  \]
  so, by the induction hypothesis, $\bar J\in\SLev_q(n-q+1)$, and, since $J=(\bar J\circ_{n-q+1}(\emptyset,[q]))\cdot \sigma^{-1}$, we conclude that $J\in\SLev_q(n)$.
\end{proof}

\begin{proof}[Proof of Theorem \ref{theo:Lev}]
  By definition, as a suboperad of $\Pi$, $\SLev_q$ is generated by an operation $\star$ of arity $p$ fixed by the action of $\∑_q$. As a partition, recall from Example \ref{ex:comppart} that $(\star\circ_2\star)\circ_1\star=(\emptyset,\{2q+1,\dots,3q-2\},\{1,\dots,2q\})$, so $\star$ satisfies the relation
  \begin{equation}
    (\star\circ_2\star)\circ_1\star=(\star\circ_2\star)\circ_1\star \cdot (q\ q+1).\tag{$***$}\label{*}
  \end{equation}
  Let $\SMagCom_q$ be the free operad generated by an arity $q$ operation fixed by the action of $\∑_q$. Let $\varphi:\SMagCom_q\to\SLev_q$ be the unique operad morphism sending $\mu$ to $\star$, and $\sim$ be the equivalence relation in $\SMagCom_q$ given by $t\sim t'$ if and only if $\varphi(t)=\varphi(t')$ To show that all the relations satisfied by $\star$ are generated by the relation \eqref{*}, we need to show that, in $\SMagCom_q$, $\sim$ is implied by (or contained in) the relation $\approx$ generated by $(\mu\circ_2\mu)\circ_1\mu \approx (\mu\circ_2\mu)\circ_1\mu \cdot (q\ q+1)$ under operadic composition and action of the symmetric groups.

  We will use the description of the free operad using the tree module from, for example, \cite{LV}, section 5.4. Note that this reference concerns operads in vector spaces, but can be readily seen as a statement on the linearisation of set operads. Let $M$ be the symmetric sequence concentrated in arity $q$ with $M(q)$ being the trivial representation. Denote by $\mathcal T_0M=I$ the unit operad, $\mathcal T_1M=I\coprod M$, and by induction, for all $n>0$, $\mathcal T_nM=I\coprod(M\circ\mathcal T_{n-1}M)$. Denote by $\mathcal T M=\bigcup_n\mathcal T_nM$. According to \cite{LV}, Theorem 5.4.2, there is an operad structure on $\mathcal TM$, and with this structure, $\mathcal T M$ is isomorphic as an operad to $\SMagCom_q$. Moreover, since $M$ is concentrated in arity $q$, it is easy to prove by induction that for all $n$, $\mathcal T_n M$ is concentrated in arities $1+k(q-1)$ for $0\le k\le n$. So $\mathcal TM$ is concentrated in arities $1+k(q-1)$ for $k\in\N$.

  Still following \cite{LV}, elements of $\mathcal T_nM$ are (non-planar) labelled rooted trees of at most $n$ levels with two types of vertices: the leaves, and the inner vertices, which are indexed by $\mu$ and have $q$ ingoing edges. The element $\mu\in\SMagCom_q(q)$ is identified with the $q$-corolla, that is, the tree with one single inner vertex and $q$ leaves. Elements of $\mathcal T M(n)$ are such trees with $n$ leaves labelled 1 to $n$. If $t\in\mathcal TM(n)$ is such a tree, then $\varphi(t)\in\SLev_q(n)\subset\Pi(n)$ is a partition $\varphi(t)=(\varphi(t)_i)_{0\le i\le s}$ of $[n]$ where $\varphi(t)_i$ is the set of labels of leaves of $t$ at height $i$. Since $\Pi(n)$ is the set of ordered partitions of $[n]$, this implies that two trees $t,t'$ have same image under $\varphi$ if and only if for all $i$, $t$ and $t'$ have the same set of labels for their leaves of height $i$.
  
  Let $(t,t')$ be such a pair of trees, with $n$ leaves. Denote by $s$ the height of $t$, and for all $i\in\{0,\dots,s\}$ let $\alpha_i=|\varphi(t)_i|$. There exists a $\sigma\in\∑_n$ such that 
\[
  \varphi(\sigma t)_s=\{1,\dots,\alpha_s\},\quad \varphi(\sigma t)_{s-1}=\{\alpha_s+1,\dots,\alpha_s+\alpha_{s-1}\}, \dots,
\]
\[
  \varphi(\sigma t)_i=\{\alpha_s+\alpha_{s-1}+\dots+\alpha_{i+1}+1,\dots,\alpha_s+\dots+\alpha_i\}.
\]
  This also implies that $\varphi(\sigma t)_i=\{\alpha_s+\dots+\alpha_{i+1}+1,\dots,\alpha_s+\dots+\alpha_i\}$. So $\sigma t$ and $\sigma t'$ are trees such that the leaves of height $i$ are labelled $\{\alpha_s+\dots+\alpha_{i+1}+1,\dots,\alpha_s+\dots+\alpha_i\}$. There exists a permutation $\rho\in \∑_{\alpha_s}\times\dots\times \∑_{\alpha_0}$ such that $\sigma t'=\rho\sigma t$. 
  
  The fact that $\varphi$ is compatible with actions of the symmetric groups then implies that the equivalence relation $\sim$ is generated $t\sim\rho t$ where the leaves of $t$ of height $i$ are labelled $\{\alpha_s+\dots+\alpha_{i+1}+1,\dots,\alpha_s+\dots+\alpha_i\}$, and $\rho\in\∑_{\alpha_s}\times\dots\times\∑_{\alpha_0}$. Since all inner vertices have $q$ incoming edges, $\alpha_s=lq$ for an integer $l>0$. Without loss of generality, we can assume that we have chosen $\sigma$ such that the leaves of $t$ labelled $\{jq+1,\dots,(j+1)q\}$ are the leaves of the same internal vertex of height $s-1$ for all $j\in[l]$.

  We will now prove, by recurrence on the height $s$, that $t\approx\rho t$. Denote by $\rho=\rho_s\times\rho_{s-1}\times\dots\times\rho_0$ where $\rho_i\in\∑_{\alpha_i}$.

  Let $\bar t$ be the tree obtained from $t$ by removing the ingoing edges of all the inner vertices of height $s-1$ (there are $l$ of them), relabelling these vertices with $\{1,\dots,l\}$, and relabelling the leaf in $t$ labelled $i$ by $i-\alpha_s+l$ for all $i>\alpha_s$. We have described an element $\bar t\in \mathcal TM(n-\alpha_s+l)$ of height $s-1$, such that $t=\bar t(\mu^{\times l},1_{\SMagCom_q}^{\times n-\alpha_s})$ (this is $\bar t$ composed with $\mu$ in the $l$ first inputs). Denote by $\hat t=(\id_{[l]}\times\rho_{s-1}\times\dots\times\rho_{0})\bar t$. by the induction hypothesis, we know that $\bar t \approx \hat t$, so, composing by $\mu$ in the $l$ first input, $t\approx(\id_{[\alpha_s]}\times\rho_{s-1}\times\dots\times\rho_{0})t$.

  It remains to prove that $\rho t\approx (\id_{[\alpha_s]}\times\rho_{s-1}\times\dots\times\rho_{0})t$, or equivalently, that $\rho_{s} t\approx t$, where we identified $\rho_s$ with $\rho_s\times\id_{[\alpha_{s-1}]}\times\dots\times\id_{\alpha_{0}}$. Since $\∑_{\alpha_s}$ is generated by transpositions, we can suppose that $\rho_s$ is a transposition $(i_1\ i_2)$.

  If $i_1,i_2\in \{jq+1,\dots,(j+1)q\}$ for a certain $j\in[l]$, then since $\mu$ is fixed by $\∑_q$, $\rho_s t=t$. Suppose now that there exists $j_1,j_2\in[l]$ with $j_1\neq j_2$, such that $j_1q+1\le i_1\le(j_1+1)q$ and $j_2q+1\le i_2\le(j_2+1)q$. Since $\mu$ is fixed by $\∑_q$, we can assume that $i_1=(j_1+1)q$ and $i_2=j_2q+1$

  Since $\bar t$ is a tree of height $s-1$ then $l$ is again a multiple of $q$, $l=ql'$, and since we have assumed that $l>1$, $l'>0$. 

  by the induction hypothesis, $(j_2\ 2)(j_1\ 1)\bar t\approx \bar t$. We can again suppose, without loss of generality, that we have chosen $\sigma$ such that the leaves of $(j_2\ 2)(j_1\ 1)\bar t$ labelled $\{jq+1,\dots,(j+1)q\}$ are the leaves of the same internal vertex of height $s-2$ for all $j\in[l']$. Composing with $\mu$ in the $l$ first inputs then yields $\tau  t\approx t$, where $\tau$ is a block permutation, where the blocks have size $q$, and $\tau$ transpose the $j_1$-th block with the first block and the $j_2$-th block with the second block. Then, $\tau\rho\tau^{-1}$ is the transposition $(q\ q+1)$. I pretend that $(q\ q+1) \tau t\approx\tau t$.

  Let $\bar s$ be the tree obtained from $(j_2\ 2)(j_1\ 1)\bar t$ by collapsing the leaves labelled $1$ to $q$, relabelling the new leaf of level $s-2$ by $1$, and relabelling all other leaves labelled $i$ in $(j_2\ 2)(j_1\ 1)\bar t$ by $i-q+1$. Then $\bar s$ is a new element in $\mathcal TM(n-\alpha_s+l)$ such that $(j_2\ 2)(j_1\ 1)\bar t=\bar s\circ_1\mu$. Denote by $s=\bar s(1_{\SMagCom_q},\mu^{\times l-1},1_{\SMagCom_q}^{\times n-\alpha_s})$ (that is $\bar s$ composed by $\mu$ in the inputs $2,3\dots,l$). We then have:
  \[
    \tau t=s\circ_1\left(\dots\left((((\mu\circ_2 \mu)\circ_1 \mu)\circ_{2q+l-2}\mu)\circ_{2q+l-3}\mu\right)\dots \circ_{2q+1}\mu\right),
  \]
  and
  \[
    (q\ q+1) \tau t=s\circ_1\left(\dots\left(((((\mu\circ_2 \mu)\circ_1 \mu)(q\ q+1))\circ_{2q+l-2}\mu)\circ_{2q+l-3}\mu\right)\dots \circ_{2q+1}\mu\right).
  \]
  Since $\approx$ is generated by $(\mu\circ_2 \mu)\circ_1 \mu)\approx((\mu\circ_2 \mu)\circ_1 \mu)(q\ q+1)$, we have proven that $(q\ q+1) \tau t\approx\tau t$, so, $\tau^{-1}((q\ q+1) \tau t\approx\tau t)$, that is, $\rho_st\approx t$.
\end{proof}

\part{Unstable modules, unstable algebras}
We now turn to unstable modules and algebras over the Steenrod algebra. The main result of this part of the article is Theorem \ref{theo:freeunst}, which characterise certain free unstable algebras over operads.

\section{The Steenrod algebra, unstable modules}\label{sec:Steenrodalg}
    The Steenrod algebra is a central object in homotopy theory and algebraic topology. It was built to represent the natural operations in cohomology that are stable under the suspension operation. As such, the cohomology of any topological space can be seen as a module over the Steenrod algebra. Unstable modules are modules over the Steenrod algebra satisfying an additional property which models the behaviour of these cohomology modules.
    
    In this section, we describe the Steenrod algebra $\A(q)$ of reduced $q$-th powers (without the Bockstein operator). This algebra, which is a sub-bialgebra of the classical Steenrod algebra, can be presented in various equivalent ways, including the presentation given by \cite{JM}, and more recently \cite{K1}. For the sake of clarity of our argument, we will give our own presentation, which is again equivalent. The statements for $\A(q)$ and its unstable modules are well known to easily lead to results on the Steenrod algebra and its unstable modules (see for example \cite[appendix A.1]{LZ}).
    
    We recall that, throughout this paper, $\F$ is a field of characteristic $p$ and order $q=p^\alpha$. For the definitions of the category of unstable modules over this algebra, we will rely on the very detailed textbook \cite{LS}. The definitions in \cite{LS} concern unstable modules over the Steenrod algebra, but extend readily to the unstable modules over the sub-bialgebra of reduced powers.
    \begin{defi}[The Steenrod algebra of reduced powers]
    The \textbf{Steenrod algebra of $q$-th reduced powers} (without Bockstein operator, with grading divided by 2 if $p>2$), is the bialgebra $\A(q)$ generated, as a unital associative algebra, by elements $P^i$ for all $i>0$ of degree $(q-1)i$, satisfying relations called \textbf{Adem relations}. The coproduct in $\A(q)$ is given by:
    \[
        \Delta(P^i)=\sum_{j+k=i}P^j\otimes P^k,
    \]
    where $P^0$ is understood to be the unit of $\A(q)$.
    \end{defi}
    \begin{defi}[Unstable modules over the Steenrod algebra, reduced modules, connected modules, the suspension]\label{def:unstredconnsus}
       A graded $\A(q)$-module $M=(M^i)_{i\in\N}$ is \textbf{unstable} if for all $x\in M^i$, $P^jx=0$ whenever $j>i$. Let $\U(q)$ denote the full subcategory of $\A(q)_{\mod}$ with objects the unstable $\A(q)$-modules.

       For any $\A(q)$-module $M$, and any element $x\in M^i$, let $P_0x=P^ix$. This induces an endomorphism $P_0:M\to M$ which doubles the degree.

       An unstable module $M$ is said to be \textbf{reduced} if $P_0$ is injective.

       An unstable module $M$ is said to be \textbf{connected} if $M^0=\0$.

       Let $M$ be an $\A(q)$-module. The \textbf{suspension} of $M$ is the $\A(q)$-module $\Sigma M$ such that $(\Sigma M)^0=0$ and $(\Sigma M^i)=M^{i-1}$ for all $i>0$, and such that $P^i (\sigma x)=\sigma(P^ix)$, where $\sigma x\in\Sigma M^{j+1}$ denotes the element representing $x\in M^j$.
   \end{defi}
   \begin{rema}
       Our definition of a reduced unstable module differs from the classical notion of reduced unstable module over the Steenrod algebra \cite[p. 47]{LS}. However, {\cite[Lemma 2.6.4]{LS}} shows that the two notions are equivalent.

       If $M$ is unstable, $\Sigma M$ is also unstable. This induces a functor $\Sigma:\U(q)\to\U(q)$.
   \end{rema}
   
   \begin{prop}[{\cite[p. 28]{LS}}]\label{prop:adjunction}
       The suspension functor $\Sigma:\U(q)\to\U(q)$ admits a left adjoint $\Omega:\U(q)\to\U(q)$. Moreover, for all unstable modules $M$,
       \[\Sigma \Omega M= M/P_0M,\]
       where $P_0M\subseteq M$ denotes the unstable submodule which is the image of the top-square.
   \end{prop}
\begin{defi}[Admissible sequences, excess, free unstable modules, see {\cite[pp. 19, 23, 25]{LS}}]\label{def:adm-excess-free}
   Let $I=(i_1,\dots,i_k)$ be a finite sequence of integers. Then $I$ is called \textbf{admissible} if for all $h\in\{1,2,\dots,k-1\}$, one has~$i_h\ge q i_{h+1}$, where $q$ is the order of the base field $\F$. The \textbf{excess} of an admissible sequence is the positive integer 
   \[
     e(I)=(i_1-qi_2)+(i_2-qi_3)+\dots+(i_{k-1}-qi_k).
   \]
  Let $P^I$ denote the product $P^{i_1}\cdots P^{i_k}$ in~$\A(q)$.

  For $n\in\N$, let $F(n)$ denote the \textbf{free unstable module} generated by one element $\iota_{n}$ of degree $n$. One has $\Hom_{\U(q)}(F(n),M) \cong M^{n}$.
\end{defi}
  The following assertions are all consequences of the definition of $F(n)$:
  \begin{itemize}
    \item \label{lemFn} The object $F(n)$ is isomorphic to:
    \[\Sigma^{n}\A(q)/(P^I \colon I \mbox{ admissible, and }e(I)>n).\]
    Therefore, the set of $P^I\iota_{n}$ where $I$ satisfies~$e(I)\le n$ is a graded vector space basis for $F(n)$. In particular, the unstable module $F(1)$ has a basis given by the set $\{j_k\}_{k\in\N}$, where $j_k=P^{q^{k-1}}\dots P^{1}\iota_1\in F(1)^{q^k}$, and $j_0=\iota_1$.
    \item\label{omegaFn} As a consequence of Proposition \ref{prop:adjunction}, for all $n>0$, there is an isomorphism in $\U(q)$:
    \[
      \Omega F(n)\cong F(n-1).
    \] 
    Indeed, for $M$ an unstable module, there is a one-to-one correspondence (natural in $M$):
    \begin{align*}
      \Hom_{\U(q)}(\Omega F(n),M)&\cong \Hom_{\U(q)}(F(n),\Sigma M)\\
      &\cong \left(\Sigma M\right)^{n}\\
      &\cong M^{n-1}\\
      &\cong\Hom_{\U(q)}(F(n-1),M).
    \end{align*}
    \end{itemize}

\begin{lemm}\label{lemm:sigmaomegaF(n)}
  There is an isomorphism $\Sigma\Omega F(n)\cong \Sigma F(n-1)$.
\end{lemm}
\begin{proof}
  $F(n)$ is isomorphic to:
    \[\Sigma^{n}\A(q)/(P^I \colon I \mbox{ admissible, and }e(I)>n).\]
  Note that the image of $P_0$ in this module is spanned by the $P^I$ with~$e(I)=n$. So, $\Sigma\Omega F(n)$ is isomorphic to:
  \[\Sigma^{n}\A(q)/(P^I \colon I \mbox{ admissible, and }e(I)>n-1),\]
  and this is isomorphic to~$\Sigma F(n-1)$.
\end{proof}

    \section{Algebras in unstable modules}\label{sec:unstalg}
    In this very short section, we study algebras over an operad in the category of $\A(q)$-modules, and in the category $\U(q)$, using the notions introduced in Section \ref{sec:modbialg}.
\begin{prop}
  A $\P$-algebra in $\A(q)$-module (see Definition \ref{def:PalginBmod}) is a $\P$-algebra $M$ endowed with an action of $\A(q)$ that satisfies the (generalised) Cartan formula, that is, for all $\mu\in\P(n)$, $(x_i)_{1\le i\le n}\in M^{\times n}$,
  \[P^i \mu(x_1,\dots,x_n)=\sum_{i_1+\dots+i_n=i}\mu(P^{i_1}x_1,\dots,P^{i_n}x_n).\]
\end{prop} 
\begin{proof}
    This is an easy verification, noting that the $n-1$-th iterated coproduct $\Delta^{n-1}:\A(q)\to \A(q)^{\otimes n}$ sends $P^i$ to $\sum_{i_1+\dots+i_n=i}P^{i_1}\otimes\cdots\otimes P^{i_n}$.
\end{proof}
\begin{defi}
    A $\P$-algebra in $\A(q)$-modules which is unstable as an $\A(q)$-module is called a \textbf{$\P$-algebra in unstable module}.

    Let $\P_{\alg}^{\U(q)}$ denote the full subcategory of $\P_{\alg}^{\A(q)}$ with objects the $\P$-algebras in unstable modules.
\end{defi}
\begin{prop}
    If $M$ is an unstable module, then $S(\P,M)$ is a $\P$-algebra in unstable modules. The forgetful functor $\P_{\alg}^{\U(q)}\to \U(q)$ admits as left adjoint the functor $S(\P,-):\U(q)\to\P_{\alg}^{\U(q)}$.
\end{prop}
\begin{proof}
    This is a straightforward generalisation of \cite[Proposition 2.5]{SI2}, which covers the case $p=2$.
\end{proof}

    \section{Unstable algebras over the Steenrod algebra}\label{sec:unst}
   The classical notion of an unstable algebra over the Steenrod algebra is that of a unital, commutative, associative algebra $A$, which is also an unstable module over the Steenrod algebra, and satisfies the Cartan formula and the additional relation $P_0x=x^q$, also called instability. In our terminology, this is an object of $\uCom_{\alg}^{\U(q)}$ satisfying the instability relation. A variant of this notion is the notion of unstable level algebra due to \cite{CL}. Here we generalise the notion of unstable algebra to algebras over any operad $\P$ equipped with a $q$-ary operation $\star$.

   In this section we introduce the Brown--Gitler and Carlsson modules and algebra. These objects come equipped with a $q$-ary operation, but are not unstable algebras in the classical sense: in the Brown-Gitler algebra and Carlsson algebra for example, $P_0x$ is not equal to $x^q$, but instead, $P_0x=\phi(x)^q$ where $\phi$ is a certain endomorphism. These provide examples of unstable $q$-level algebras, where this notion has been introduced in Section \ref{sec:Lev}. In Section \ref{sec:BGC}, we will see that some of those objects are in fact free, as unstable algebras over certain operads.

   \begin{defi}[$\star$-Unstable $\P$-algebras]
       Let $\star\in\P(q)$. A \textbf{$\star$-unstable $\P$-algebra} is a $\P$-algebra $A$ in $\U(q)$ such that $P_0a=\star(a,\dots,a)$ for all $a\in A$.

       The $\star$-unstable $\P$-algebras form a full subcategory of $\P_{\alg}^{\U(q)}$ which is denoted by $\K_{\P}^\star$.

       As an example, an \textbf{unstable $q$-level algebra} is a $\star$-unstable $\Lev_q$-algebra, where $\Lev_q$ is defined in \ref{defi:Lev} and $\star\in\Lev_q$ is the operadic generator.
   \end{defi}

   When $\P=\uCom$ and $\star=X_q$ (see Example \ref{ex:operads}), we recover the classical notion of unstable algebras from, for example, \cite{K3}.
   
   Let us give some examples of $\star$-unstable $\P$-algebras that appear in literature for different operads $\P$ and operations $\star$. The classic example is the following:
   \begin{ex}
       Let $X$ be a topological space. Then, the cohomology of $X$ with coefficients in $\F_p$ inherits an unstable $\A(p)$-action (see, for example, \cite[Theorem 1.1.1.]{LS}). The cup-product endows $H^*(X,\F_p)$ with the structure of a unital, commutative, associative algebra in $\U(p)$, with the additional relation $P_0x=x^p$. In other words, $H^*(X,\F_p)$ is an $X_p$-unstable $\uCom$-algebra (see Example \ref{ex:operads} for the definition of $\uCom$ and $X_p$).
   \end{ex}
   We now introduce unstable modules which are analogues of the Brown-Gitler and Carlsson modules, following the notation of \cite{K3}. These modules form algebras which, as we show, are unstable over certain operads. We choose to define these objects using their algebraic structures. They are identified with injective objects, as is shown in \cite{K3}. \cite{BG,HM,GC}, and \cite{LS} for a textbook account.
   \begin{defi}[Brown--Gitler modules and algebra, Carlsson modules and algebra]\label{def:BGC}
       The \textbf{Brown--Gitler algebra} $J$ is the $\uCom$-algebra in $\U(q)$ whose underlying $\uCom$-algebra is the polynomial algebra $\F[x_i,i\in\N]$, with $|x_i|=1$, endowed with the (unstable) action of $\A(q)$ induced on generators by:
\begin{equation}
  P^jx_{i}=\left\{\begin{array}{ll}
    x_i,&\mbox{if }j=0,\\
    x_{i-1}^p,&\mbox{if }j=1,\\
    0,&\mbox{otherwise,}
  \end{array}\right.
\end{equation}
  where we set $x_{-1}=0$.

  The Brown--Gitler algebra is equipped with a second grading $w$ called the \textbf{weight}, which is additive with respect to multiplication, and such that $w(x_i)=q^{i}$.

  The \textbf{Brown--Gitler module of weight $n$} is the submodule $J(n)\subseteq J$ of homogeneous elements of weight $n$.

  The \textbf{Carlsson algebra} $K$ is the $\uCom$-algebra in $\U(q)$ whose underlying $\uCom$-algebra is the polynomial algebra $\F[x_i,i\in\Z]$, with $|x_i|=1$, endowed with the (unstable) action of $\A(q)$ induced on generators by:
\begin{equation}
  P^jx_{i}=\left\{\begin{array}{ll}
    x_i,&\mbox{if }j=0,\\
    x_{i-1}^p,&\mbox{if }j=1,\\
    0,&\mbox{otherwise.}
  \end{array}\right.
\end{equation}

  The Carlsson algebra is equipped with a second grading $w$ called the \textbf{weight}, additive with respect to multiplication, with $w(x_i)=q^{i}$. Note that this weight has range $\N[\frac1q]$.

  The \textbf{Carlsson module of weight $n$} is the submodule $K(n)\subseteq K$ of homogeneous elements of weight $n$.
   \end{defi}
   \begin{rema}\label{rema:Jlim}
       Consider the algebra endomorphism $\phi:J\to J$ sending $x_i$ to $x_{i-1}$. This induces morphisms of unstable modules $J(qn)\to J(n)$. For $n\in\N$, the Carlsson module $K(n)$ is isomorphic to the limit in $\U(q)$ of the diagram:
       \[
        \diag{J(n)&J(qn)\ar[l]_{\phi}&J(q^2n)\ar[l]_{\phi}&\cdots\ar[l]}.
       \]
       This implies that there are isomorphisms $K(n)\cong K(qn)$. In fact, these isomorphisms are induced by the algebra isomorphism $\phi:K\to K$ sending $x_i$ to $x_{i-1}$.
   \end{rema}
From the very definition of these objects, it is clear that neither the Brown--Gitler algebra nor the Carlsson algebra are $X_p$-unstable $\uCom$ algebra. However, we have the following:
\begin{lemm}
    The Brown--Gitler algebra and the Carlsson algebra are equipped with a $\star$-unstable $\Lev_q$-algebra structure, where $\star$ and $Lev_q$ are defined in \ref{defi:Lev}.
\end{lemm}
\begin{proof}
    Note that $\phi:J\to J$ and $\phi:K\to K$ defined above are endomorphisms of $\uCom$-algebras in $\U(q)$. So, $J$ and $K$, endowed with $\phi$, form two $\uCom\circ\D$-algebras, where $\uCom\circ\D$ is defined in \ref{rema:uComcircD} and $D$ acts by $d=\phi$. Since $\uCom\circ\D$ is isomorphic to $\F\Pi$ (see \ref{rema:uComcircD}), and $\Lev_q$ is defined as a suboperad of $\F\Pi$ (see \ref{defi:Lev}), it ensues that $J$ and $K$ can be equipped with the structure of a $\Lev_q$-algebra in $\U(q)$ by restriction of structure. Since $\Lev_q$ is generated by the element $\star=(\emptyset,[q])\in\F\Pi(q)$, which corresponds to the operation $(X_q;d,\dots,d)$ under the isomorphism $\F\Pi\cong\uCom\circ\D$, these $\Lev_q$-algebra structures on $J$ and $K$ are defined by:
    \[
    \star(a_1,\dots,a_q)=\phi(a_1)\cdots \phi(a_q),
    \]
    for all elements $a_1,\dots,a_q$ of $J$, or of $K$. To show that $J$ and $K$ are $\star$-unstable, it suffices to check that $P_0x_i=\star(x_i,\dots,x_i)$ on the generators $x_i$. But $P_0x_i=x_{i-1}^{q}$, and $\star(x_i,\dots,x_i)=\phi(x_i)^q=x_{i-1}^q$.
\end{proof}
\begin{rema}\label{rema:JKLevalg}
    For all $n$, the $\Lev_q$-operation $\star$ defined above on $J$ and $K$ stabilise the modules $J(n)$ and $K(n)$. Indeed, note that $\phi$ divides the weight by $q$, so, if $a_1,\dots,a_q$ all have weight $n$, $\star(a_1,\dots,a_q)=\phi(a_1)\cdots \phi(a_q)$ is a product of $q$ elements of weight $n/q$.
\end{rema}
   \section{Free unstable algebras}\label{sec:freeunst}
   
    The aim of this section is to prove our main result, Theorem \ref{theo:freeunst}, which identifies free unstable algebras generated by certain unstable modules to free algebras. More precisely, given an operad $\P$ equipped with a $\P$-central operation $\star\in\P(q)^{\∑_q}$ (see Definition \ref{def:centop}), we will build a commutative diagram of categories:
\begin{equation}\label{diag:complete}\tag{D2}
    \diag{\Vect\ar[d]_{\D\otimes -}\ar[r]^{S(\P,-)}&\P_{\alg}\\\D_{\mod}\ar[r]^{\psi_!}&\P_{\alg}^{\D}\ar[u]\\
\U(q)\ar[u]^{\varphi}\ar[r]^{K_\P^{\star}}&\K_{\P}^{\star}\ar[u]_{\varphi}},
\end{equation}
where the top square is the diagram \ref{diag:freecomp} from Corollary \ref{coro:freecomp}, and where the bottom functor, $\psi_!:\U(q)\to \K^{\star}_{\P}$, is left adjoint to the forgetful functor $\K^{\star}_{\P}\to \U(q)$. We will then show that the image of certain unstable modules (namely, connected, reduced unstable modules) under $\varphi:\U(q)\to \D_{\mod}$, are in the essential image of $D\otimes-:\Vect\to\D_{\mod}$.
    
For any unstable module $M$, denote by $\varphi M$ the $\D$-module whose underlying vector space is $M$ and such that $dx=P_0M$. This induces a functor $\varphi:\U(q)\to \D_{\mod}$.

   \begin{lemm}\label{lemm:topsqmor}
       Let $A$ be a $\P$-algebra in $\U(q)$. Then, $P_0$ is a $\P$-algebra endomorphism of $A$. In other words, the functor $\varphi:\U(q)\to\D_{\mod}$ extends into a functor $\varphi_{\P}:\P_{\alg}^{\U(q)}\to\P_{\alg}^{\D}$, which restricts to a functor $\bar\varphi_{\P}:\K_{\P}^*\to \P_{\alg}^{\D}$.
   \end{lemm}
   \begin{proof}
       This is a straightforward generalisation of \cite[Corollary 2.7]{SI2}, which covers the case $q=2$.
   \end{proof}
   
   Recall that, for any $\star\in\P(q)$, we built, in Section \ref{sec:centop}, a functor $\psi_!:\D_{\mod}\to \P_{\alg}$ that sends a $\D$-module $M$ to the free $\P$-algebra over $M$ satisfying $dx=\star(x,\dots,x)$ for all $x\in M$. 
   \begin{prop}\label{prop:KPstar}
       Suppose that $\star\in\P(q)$ is fixed under the action of $\∑_q$. The functor $\psi_!\circ\varphi:\U(q)\to \P_{\alg}$ extends into a functor $K_\P^\star:\U(q)\to\P_{\alg}^{\U(q)}$.

       Additionally, if $\star$ is $\P$-central (see Definition \ref{def:centop}), then $K_\P^\star$ restricts into a functor $K_\P^\star:\U(q)\to\K_{\P}^{\star}$.
   \end{prop}
   \begin{proof}
       For any unstable module $M$, the $\P$-algebra $\psi_!\circ\varphi(M)$ is defined as $S(\P,M)/(P_0^kx-\star_k(x,\dots,x),x\in M,k\in\N)$. Note that $S(\P,M)$ is also the free $\P$-algebra in $\U(q)$ generated by $M$. In other words, $S(\P,M)$ can be equipped with an unstable $\A(q)$-action by:
       \[
        P^i\cdot (\mu;x_1,\dots,x_n)=\sum_{j_1+\dots+j_n=i}(\mu;P^{j_1}x_1,\dots,P^{j_n}x_n).
       \]
       To make sure that this induces an unstable $\A(q)$-action on $S(\P,M)/(P_0^kx-\star_k(x,\dots,x),x\in M,k\in\N)$, it suffices to check that the set $X=\{P_0^kx-\star_k(x,\dots,x),x\in M,k\in\N\}$ is stable under the action of $\A(q)$. For all $k,i\in \N$ and $x\in M$, one can show that
       \[
        P^iP_0^kx=\begin{cases}
            P_0^kP^{\frac{i}{q^k}}x,&\mbox{if }q^k|i,\\
            0,&\mbox{otherwise.}
        \end{cases} 
       \]
       See for example, \cite[Section 1.7]{LS}. On the other hand, one has
  \begin{equation*}
    P^i(\star_k;x,\dots,x)=\sum_{j_1+\dots+j_{q^k}=i}(\star_k;P^{j_1}x,\dots, P^{j_{q^k}} x).
  \end{equation*}
  Fix $j'_1\le\dots\le j'_{q^k}\in\N$ such that $j'_1+\dots+j'_{q^k}=i$. Suppose that we have:
  \begin{multline*}
        j'_1=j'_2=\dots=j'_{r_1}<j'_{r_1+1}=\dots=j'_{r_1+r_2}\\
        <\dots<j'_{r_1+\dots+r_{s-1}+1}=\dots=j'_{r_1+\dots+r_s},
  \end{multline*}
  with $r_1+\dots+r_s=q^k$, that is, $r_1+\dots+r_s$ is the coarsest partition of the integer $q^k$ such that $j'_l=j'_{l'}$ if and only if there exists $m\in[s]$ such that $r_1+\dots+r_{m-1}<l,l'<r_1+\dots+r_m$. Since $\star$ is fixed under the action of $\∑_q$, the term $\star(P^{j'_1}x,\dots, P^{j'_q} x)$ appears in the above sum exactly $\frac{q!}{r_1!\dots r_s!}$ times. This integer is divisible by $q$ unless $s=1$, in which case, $r_1=q$, and the condition $j'_1+\dots+j'_q=i$ implies that $i\equiv 0\ [q]$ and $j'_1=\dots=j'_q=\frac{n}{q}$.

  This implies that
  \begin{equation*}
    P^i(\star;x,\dots,x)=\begin{cases}
      \left(\star;P^{i/q}x,\dots,P^{i/q}x\right),&\mbox{if } i\equiv 0\ [q]\\
      0,&\mbox{ otherwise.}
    \end{cases}
  \end{equation*}

  Since both $0$ and $P_0P^{\frac{i}{q}}x-\left(\star;P^{i/q}x,\dots,P^{i/q}x\right)$ belong to $X$, we conclude that $X$ is stable under the action of $\A(q)$. So we defined an unstable $\A(q)$-action on $\psi_!\circ\varphi(M)$ which is compatible with the $\P$-algebra structure.

  To prove the second assertion of our proposition, suppose now that $\star$ is $\P$-central. We want to show that $\psi_!\circ\varphi(M)$, with the above $\A(q)$-action, is $\star$-unstable. In other words, we want to show that, for all $\mu\in\P(n)$, $x_1,\dots,x_n\in M$, one has:
  \[
    P_0(\mu;x_1,\dots,x_n)\sim_{X}\star((\mu;x_1,\dots,x_n)^{\times q})=(\star(\mu^{\times q});x_1,\dots,x_n,\dots,x_1,\dots,x_n),
  \]
  where $\sim_{X}$ denotes the equivalence relation defined by the quotient by the $\P$-ideal generated by $X$. From Lemma \ref{lemm:topsqmor}, we deduce that:
  \[
    P_0(\mu;x_1,\dots,x_n)=(\mu;P_0x_1,\dots, P_0x_n).
  \]
    But, one has:
    \[
    (\mu;P_0x_1,\dots, P_0x_n)\sim_{X}(\mu,\star(x_1^{\times q}),\dots,\star(x_n^{\times q}))
    \]
    This last element is equal to:
    \[
    (\mu(\star^{\times n});x_1^{\times q},\dots,x_n^{\times q}).
    \]
    Using the fact that $\star$ is $\P$-central, $\mu(\star^{\times n})=\star(\mu^{\times q})\cdot\sigma_{q,n}^{-1}$, so, 
    \[
(\mu(\star^{\times n});x_1^{\times q},\dots,x_n^{\times q})=(\star(\mu^{\times q});x_1,\dots,x_n,\dots,x_1,\dots,x_n).
    \]
    This concludes our proof.
   \end{proof}
   \begin{prop}
       Let $\star\in\P(q)$ be a $\P$-central operation fixed under the action of $\∑_q$. Then $K_\P^\star(M)$ is the free $\star$-unstable $\P$-algebra over $M$. In other words, $K_\P^\star:\U(q)\to\K_{\P}^{\star}$ is a left adjoint to the forgetful functor $\K_{\P}^{\star}\to \U(q)$.
   \end{prop}
   \begin{proof}
       This is a somewhat straightforward generalisation of \cite[Proposition 6.7]{SI2}. Since we are using different constructions in this article, let us give the detailed proof.
       
       Let $M$ be an object in $\U(q)$, $A$ an object in $\K_\P^{\star}$. We show that there is a bijection $\Hom_{\U(q)}(M,A)\cong \Hom_{\K_{\P}^{\star}}(K_\P^{\star}(M),A)$. 
       
    Recall from the construction of $K_{\P}^\star$ that there is a surjective morphism $S(\P,M)\to K_{\P}^\star(M)$ in $\U(q)$ (see the proof of \ref{prop:KPstar}). The unit of the monad $S(\P,-)$ provides a morphism of unstable modules $M\to S(\P,M)$. To any morphism $f:K_\P^{\star}(M)\to A$ in $\K_\P^{\star}$, we associate the morphism in $\U(q)$ which is the following composition:
    \[
    \bar f:\diag{M\ar[r]& S(\P,M)\ar[r]& K_\P^{\star}(M)\ar[r]^-{f}& A.}
    \]
    Let now $g:M\to A$ be a morphism in $\U(q)$. There is a unique morphism $g':S(P,M)\to A$ in $\P_{\alg}^{\U(q)}$. For any $x\in M$, since $g'$ is compatible with the action of $\A(q)$, $g'(P_0x)=P_0g'(x)$. Since $g'$ is compatible with the $\P$-algebra structures, $g'(\star(x,\dots,x))=\star(g'(x),\dots,g'(x))$. Finally, since $A$ is unstable, $P_0g'(x)=\star(g'(x),\dots,g'(x))$. So, for any $x\in M$, $g'(P_0x-\star(x,\dots,x))=0$, which implies that $g'$ passes to the quotient into a morphism 
    \[
    \hat g:S(P,M)/(P_0x-\star(x,\dots,x),x\in M)\to A
    \]
    in $\P_{\alg}^{\U(q)}$, which can be seen as a morphism $\hat g:K_\P^{\star}(M)\to A$. Showing that the associations $f\mapsto \bar f$ and $g\mapsto \hat g$ provide inverse bijections between $\Hom_{\U(q)}(M,A)$ and $\Hom_{\K_{\P}^{\star}}(K_\P^{\star}(M),A)$ is a straightforward verification that is left to the reader.
   \end{proof}
   \begin{lemm}\label{lemm:freetopsq}
       For any connected reduced unstable module $M$, $\varphi(M)$ is isomorphic to $D\otimes \Forget(\Sigma\Omega M)$, where $\Forget:\U(q)\to \Vect$ is the forgetful functor that extracts the underlying vector space of an unstable module. This isomorphism is not unique.
   \end{lemm}
   \begin{proof}
       The unit of the adjunction $\Sigma \dashv \Omega$ from Proposition \ref{prop:adjunction} provides a map $M\to \Sigma\Omega M= M/P_0M$. Pick any linear section $s:M/P_0M\to M$ (this is not a morphism in $\U(q)$). The fact that $M$ is connected and that $P_0$ is injective and doubles the degree then implies that $M$ is freely generated by $s(\Sigma\Omega M)$ under the action of $P_0$.
   \end{proof}
We now obtain our main result as an easy consequence of the preceding results:
   \begin{theo}\label{theo:freeunst}
       Let $\star\in\P(q)$ be a $\P$-central operation fixed under the action of $\∑_q$. For all connected reduced $M$, the underlying $\P$-algebra of $K_{\P}^\star(M)\in\K_\P^\star(M)$ is isomorphic to $S(\P,\Forget(\Sigma\Omega M))$. This isomorphism is not unique.
   \end{theo}
   \begin{proof}
       Consider the diagram of categories \ref{diag:complete} from the beginning of the present section. Corollary \ref{coro:freecomp} shows that the top square commutes. The functor $K_\P^{\star}$ has been constructed in Proposition \ref{prop:KPstar} so that the bottom square commutes.

       Let $M$ be a connected, reduced unstable module. Lemma \ref{lemm:freetopsq} shows that $\varphi(M)=D\otimes \Forget(\Sigma\Omega M)$. So, the $\P$-algebra in $\D$-modules $\varphi\circ K_{\P}^\star(M)$ is isomorphic to $\psi_!(D\otimes\Forget(\Sigma\Omega M))$. Its underlying $\P$-algebra is then isomorphic to $S(\P,\Forget(\Sigma\Omega M))$. Since the underlying $\P$-algebra of $\varphi\circ K_{\P}^\star(M)$ is also the underlying $\P$-algebra of $K_{\P}^\star(M)$, we obtain the result.
   \end{proof}
    \section{Applications}\label{sec:BGC}
        In this section, we apply our result to free unstable modules to obtain a description of the free $\star$-unstable $\P$-algebra generated by one element. This allows us identify certain Brown--Gitler modules and Carlsson modules from Section \ref{sec:unst} to free unstable algebras over certain operads.
        
       Theorem \ref{theo:freeunst} shows that, for $\star\in \P(q)$ a $\P$-central operation fixed under the action of $\∑_q$, the free $\star$-unstable $\P$-algebra generated by a connected, reduced unstable module $M$ is itself free as a $\P$-algebra, generated by the underlying vector space of $\Sigma \Omega M$. In particular, if $\Sigma \Omega M$ comes with a natural choice of linear basis, we can give a basis of the $\P$-algebra $K_{\P}^{\star}(M)$. 

       Recall from Section \ref{sec:Steenrodalg} that the free unstable module $F(n)$ over an element of degree $n$ satisfies $\Sigma\Omega F(n)\cong\Sigma F(n-1)$. In particular, the underlying vector space of $\Sigma\Omega F(1)\cong\Sigma F(0)$ is one dimensional, concentrated in arity 1. We then get: 
    \begin{prop} \label{prop:KF1}
  Let $\star\in \P(q)$ be a $\P$-central operation fixed under the action of $\∑_q$. Then, $K_\P^\star(F(1))$ is the $\P$-algebra in $\U(q)$ whose underlying $\P$-algebra is freely generated by one element $\iota_1$ of degree $1$, endowed with the unstable action of $\A(q)$ induced by:
  \[P^j\iota_1=\left\{\begin{array}{ll}
    \iota_1,&\mbox{if }j=0,\\
    \star(\iota_1,\dots,\iota_1),&\mbox{if }j=1,\\
    0,&\mbox{otherwise.}
  \end{array}\right.\]
    \end{prop}
    \begin{proof}
        This is a straightforward generalisation of \cite[Proposition 8.4]{SI2}, which treats the case $q=2$. In particular, since $\Sigma\Omega F(1)\cong\Sigma F(0)$, and applying Theorem \ref{theo:freeunst}, we know that $K_\P^\star(F(1))$ is a $\star$-unstable $\P$-algebra whose underlying $\P$-algebra is generated by one element $\iota_1$ of degree 1. Since $K_\P^\star(F(1))$ is a $\P$-algebra in $\U(q)$, it suffices to inspect the action of $\A(q)$ on $\iota_1$, and the instability relation reads $P^i\iota_1=0$ for all $i>1$. Since $K_\P^\star(F(1))$ is $\star$-unstable, we necessarily have $P^1\iota_1=P_0\iota_1=\star(\iota_1,\dots,\iota_1)$.
    \end{proof}

    \begin{ex}
        When $\P=\uCom$ and $\star=X_q$, the functor $K_{X_q}^{\uCom}$ corresponds to the functor denoted by $U$ in \cite{K4}. In this case, Theorem \ref{theo:freeunst} corresponds to a remark of Kuhn \cite[p. 4223]{K4}. In \cite[Theorem 1.6]{K4}, Kuhn identifies a large family of free unstable algebras generated by unstable modules defined as representations of symmetric powers. This contains, for example, the computation of the mod 2 cohomology of Eilenberg--MacLane spaces of finite elementary abelian 2-groups \cite[Remark 1.8 (2)]{K4}.
    \end{ex}
    \begin{prop}[Compare with {\cite[Proposition 9.9 and 9.10 ]{SI2}}]\label{prop:K1}\item
  \begin{enumerate}[1)]
      \item The Carlsson module of weight one $K(1)$ (see Definition \ref{def:BGC}) with its $\Lev_q$ operation from Remark \ref{rema:JKLevalg}, is isomorphic to the free $\star$-unstable $\Lev_q$-algebra generated by~$F(1)$.
      \item For all $s\ge 1$, the Brown--Gitler module $J(q^s)$, with its $\Lev_q$ operation from Remark \ref{rema:JKLevalg}, is isomorphic to the free $\star$-unstable $T_s\Lev_q$-algebra (see \ref{def:TsLev}) generated by~$F(1)$.
      \item The Carlsson algebra $K$ (see Definition \ref{def:BGC}) with the multiplication of monomials and the operator $\phi$, is isomorphic to the free $(X_q;d,\dots,d)$-unstable $\uCom\circ\Dp$-algebra generated by~$F(1)$.
  \end{enumerate}
    \end{prop}
\begin{proof}\item
\begin{enumerate}[1)]
    \item Since $K(1)$ is a $\star$-unstable $\Lev_q$-algebra and $x_0\in K(1)$ has degree 1, there is a unique morphism $f:K_{\Lev_q}^{\star}(F(1))\to K(1)$ in $\K_{\Lev_q}^{\star}$ which sends the generator $\iota_1$ to $x_0$.

By definition, the vector space underlying $K(1)$ is generated by the monomials $\t=x_{i_1}^{k_1}\dots x_{i_n}^{k_n}$, of degree $k_1+\dots+k_n$, satisfying $\sum_{j=1}^nq^{i_j}=1$. Its $q$-level operation is induced by $(\t_1,\dots,\t_p)\mapsto \phi(\prod_{j=1}^p\t_j)$, where $\phi\colon K\to K$ is the algebra morphism induced by $x_i\mapsto x_{i-1}$.

  According to Proposition \ref{prop:KF1}, The free unstable $q$-level $\Lev_q$-algebra generated by an element of degree one has an underlying $\Lev_q$-algebra freely generated by one element of degree $1$. This means that, as a vector space, it is isomorphic to $\bigoplus_{n\ge 0}\Lev_q(n)/\∑_n$, with $\Lev_q(n)/\∑_n$ in degree $n$. Recall from Lemma \ref{lemm:partLev} That $\Lev_q(n)\subset \F\Pi$ is spanned by partitions $J$ of $[n]$ such that $\sum_{i=0}^s\frac{|J_i|}{q^i}=1$. Under the identification $K_{\Lev_q}^{\star}(F(1))\cong\bigoplus_{n\ge 0}\Lev_q(n)/\∑_n$, the map $f$ sends the class of a partition $I$ to $\prod_{j\in \N}x_{-j}^{|I_j|}$.
    
    A set of representatives of $\Lev_q(n)/\∑_n$ is given by the vector subspace $V(n)\subset\Lev_q(n)$ spanned by partitions $J$ such that, if $i\in J_{k}$ and $j\in J(l)$, $k<l$ implies $i<j$ (in other words, ordered partitions of $[n]$ such that the elements of $[n]$ appear in order). The generator $\iota_1\in K_{\Lev_q}^{\star}(F(1))$ is identified with the partition $(\emptyset,[q])\in V(q)$.

   Let $g:K(1)^n\to V(n)$ be the linear map sending the monomial $\t=x_{i_1}^{k_1}\dots x_{i_n}^{k_n}$ to the partition $I^{\t}$ of $k_1+\dots+k_n$ such that, for all $j\in[n]$,  $\{k_1+\dots+k_{j-1}+1,\dots,k_1+\dots+k_{j}\}\subseteq I^{\t}_{-i_j}$.

   The map $g$ is a linear inverse to $f$, proving the isomorphism.

   \item The proof is similar to 1). As a vector space, $J(q^s)$ is spanned by monomials $\t=x_{i_1}^{k_1}\dots x_{i_n}^{k_n}$, of degree $k_1+\dots+k_n$, satisfying $\sum_{j=1}^nq^{i_j}=q^s$, with $i_j\ge 0$ for all $j$. In such a monomial, no $i_j$ can be greater than $s$, otherwise, the sum will be greater than $q^s$. The $q$-level operation is induced by $(\t_1,\dots,\t_p)\mapsto \phi(\prod_{j=1}^p\t_j)$, where $\phi\colon J\to J$ is the algebra morphism induced by $x_i\mapsto x_{i-1}$ with $x_{-1}=0$. 
   
   There is a unique morphism $f:K_{\Lev_q}^{\star}(F(1))\to J(q^s)$ sending the generator $\iota_1$ to $x_s$. Under the identification $K_{\Lev_q}^{\star}(F(1))\cong\bigoplus_{n\ge 0}\Lev_q(n)/\∑_n$ given above, the map $f$ sends the class of a partition $I$ to $\prod_{j\in \N}x_{s-j}^{|I_j|}$, where $x_{s-j}=0$ whenever $j>s$. The morphism of operads $\Lev_q\to T_s\Lev_q$ induces a morphism of $\star$-unstable $\Lev_q$-algebras $K_{\Lev_q}^{\star}(F(1))\to K_{T_s\Lev_q}^{\star}(F(1))$ (see Remark \ref{rema:Morop}).
   Our morphism $f$ factors through $K_{T_s\Lev_q}^{\star}(F(1))$. Indeed, if $I\in \Lev_q^{>s}$, there is a $j>s$ such that $I_{j}\neq \emptyset$, so the image of the class of $I$ under $f$ will be a monomial containing $x_{s-j}^{|I_j|}$, which is equal to 0.

   A set of representatives of $T_s\Lev_q(n)/\∑_n$ is given by the vector subspace $V^{\le s}(n)\subset T_s\Lev_q(n)$ spanned by partitions $I$ such that, if $i\in J_{k}$ and $j\in J(l)$, $k<l$ implies $i<j$ (in other words, ordered partitions of $[n]$ such that the elements of $[n]$ appear in order), and such that $I_j=\emptyset$ if $j>s$. The generator $\iota_1\in K_{T_s\Lev_q}^{\star}(F(1))$ is identified with the partition $(\emptyset,[q])\in V^{}(q)$.

   Let $g:J(q^s)^n\to V^{\ge s}(n)$ be the linear map sending the monomial $\t=x_{i_1}^{k_1}\dots x_{i_n}^{k_n}$ to the partition $I^{\t}$ of $k_1+\dots+k_n$ such that, for all $j\in[n]$,  $\{k_1+\dots+k_{j-1}+1,\dots,k_1+\dots+k_{j}\}\subseteq I^{\t}_{-i_j}$.

   Then $g$ is a linear inverse to $f$.
   
   \item This is a straightforward generalisation of the proof of the first assertion of \cite[Proposition 9.10]{SI2}, which treats the case $q=2$. 
\end{enumerate}
\end{proof}
   % Let us now consider some operadic morphisms and the morphism of unstable algebras they induce:
    \begin{rema}\label{rema:Morop}
       Let $f:\P\to\Q$ be an operad morphism. Then, for all $\star\in\P(q)$, $f$ induces a restriction functor $f^*:\Q^{\U(q)}_{\alg}\to \P^{\U(q)}_{\alg}$. One can readily check that this restricts into a functor: $f^*:\K_{\Q}^{f(\star)}\to\K_{\P}^\star$.

       In this same setting, $f:\P\to\Q$ also induces a morphism $f_*:K_\P^\star(M)\to f^*\left( K_Q^\star(M)\right)$, natural in $M$.
    \end{rema}
    \begin{ex}
        Recall from \ref{prop:K1} that the Carlsson module of weight one $K(1)$ is isomorphic to the free $\star$-unstable $\Lev_q$-algebra generated by $F(1)$, and that the Carlsson algebra $K$ is isomorphic to the free $(X_q;d\dots,d)$-unstable $\uCom\circ\D$-algebra generated by $F(1)$. In both case, the generator of $F(1)$ is identified with $x_0$, making it easy to check that the injection $K(1)\to K$ corresponds to the map $f_*:K_{\Lev_q}^\star(F(1))\to f^*(K_{\uCom\circ\D}^{(X_q;d,\dots,d)}(F(1)))$, where $f:\Lev_q\to \uCom\circ\D$ is the injection, sending $\star$ to $(X_q;d,\dots,d)$.
    \end{ex}
    \begin{prop}
        The Carlsson module $K(1)$ is the limit of the Brown--Gitler module $J(q^s)$ in the category of $\star$-unstable $\Lev_q$-algebras.
    \end{prop}
    \begin{proof}
        Recall from Remark \ref{rema:Jlim} that $K(1)$ is isomorphic to the limit of the following diagram in $\U(q)$:
    \[
        \diag{J(n)&J(qn)\ar[l]_{\phi}&J(q^2n)\ar[l]_{\phi}&\cdots\ar[l]}.
    \]
        We have seen in \ref{prop:K1} that for all $s\ge 1$, $J(q^s)$ with its multiplication, is isomorphic $K_{T_s\Lev_q}^{\star}(F(1))$. The maps of operads $\Lev_q\to T_s\Lev_q$ makes $K_{T_s\Lev_q}^{\star}(F(1))$ a $\star$-unstable $\Lev_q$-algebra for all $s$. Since the map $\phi:J(q^{s+1})\to J(q^s)$ sends $x_{s+1}$ to $x_s$, and is compatible with the $\Lev_q$-algebra structures, it is easy to see that it corresponds to the map $K_{T_{s+1}\Lev_q}^{\star}(F(1))\to K_{T_s\Lev_q}^{\star}(F(1))$ induced by the morphism of operads $T_{s+1}\Lev_q\to T_{s}\Lev_q$.

        Therefore, the diagram above is a diagram in $\star$-unstable $\Lev_q$-algebras induced by the diagram of operads
        \eqref{diag:cof} of \ref{rema:cof}. Hence, $K(1)$ is the limit of this diagram in the category of $\star$-unstable $\Lev_q$-algebras.
    \end{proof}
    \begin{rema}\label{remaKuhnalg}
  Proposition \ref{prop:K1} shows that the Carlsson algebra $K$ can be identified with the free $d^{X_2}$-unstable $\uCom\circ\Dp$-algebra generated by $F(1)$. Restriction along the inclusion of operads $\Pi=\uCom\circ\D\to\uCom\circ\Dp$ makes $K$ a (non-free) $d^{X_2}$-unstable $\Pi$-algebra, and the quotient of $K$ by the ideal generated by $x_i$ for all $i<0$ yields the Brown--Gitler algebra $J$, seen as a $\Pi$-algebra. As a $\Pi$-ideal, this ideal is generated by the unique element $x_{-1}=d\iota_1$.

  Similarly, when $q=2$ considering $K_{\uCom\circ\Dp}^{d^{X_2}}(F(n))$ as a $d^{X_2}$-unstable $\Pi$-algebra, and quotienting by the ideal generated by $d\iota_n$, yields a description of the algebras denoted by $H^*(T(n,*),\F_2)$ in \cite{K4}. As the notation suggests, these algebras are obtained as the cohomology of a spectrum, which is related to the Eilenberg--MacLane spaces $K(V,n)$.
\end{rema}

\end{document}